\documentclass[12pt,twoside,reqno]{amsart}
\usepackage{amsmath,amsfonts,amsthm,amssymb,amscd,epsfig}
\usepackage{graphicx}
\usepackage{yhmath}
\usepackage{amssymb}
\usepackage{amsmath}
\usepackage{amsthm}
\usepackage{amsfonts}
\usepackage{amstext}
\usepackage{cite}
\usepackage[mathscr]{eucal}
\usepackage{xspace}
\usepackage{color}

\newcommand{\R}{\mathbb{R}}

\newcommand{\LL}{\mathcal{L}}
\newcommand{\GG}{\mathcal{G}}

\newcommand{\EE}{\mathcal {E}}
\newcommand{\PP}{\mathcal{P}}
\newcommand{\DD}{\mathcal{D}}
\newcommand{\CC}{\mathcal{C}}
\newcommand{\TT}{\mathcal{T}}
\newcommand{\FF}{\mathcal{F}}
\newcommand{\QQ}{\mathcal{Q}}

\newcommand{\xx}{\mathbf{x}}

\newcommand{\ff}{\mathbf{f}}
\newcommand{\kk}{\mathbf{k}}
\newcommand{\mm}{\mathbf{m}}

\newcommand{\vv}{\mathbf{v}}

\newcommand{\uu}{\mathbf{u}}
\newcommand{\vvb}{\bar{\mathbf{v}}}
\newcommand{\vb}{\bar{v}}

\newcommand{\yy}{\mathbf{y}}
\newcommand{\nn}{\mathbf{n}}

\newcommand{\ga}{\alpha}
\newcommand{\gb}{\mathbf{\beta}}
\newcommand{\gc}{\gamma}
\newcommand{\gd}{\delta}

\newcommand{\GD}{\Delta}
\newcommand{\ve}{\varepsilon}
\newcommand{\Om}{\Omega}

\newcommand{\vf}{\varphi}
\newcommand{\gl}{\lambda}
\newcommand{\GC}{\Gamma}
\newcommand{\gtt}{\boldsymbol{\tau}}
\newcommand{\gz}{\zeta}
\newcommand{\gs}{\sigma}
\newcommand{\GS}{\Sigma}

\newcommand{\dx}{\mathrm{d}\,}

\newcommand{\po}{\partial}

\newcommand{\X}{\times}

\renewcommand{\div}{\text{\rm div}\,}
\newcommand{\grad}{\nabla}

\newcommand{\dist}{\text{\rm dist}\,}

\newcommand{\bea}{\begin{eqnarray}}
\newcommand{\eea}{\end{eqnarray}}

\theoremstyle{plain}
\newtheorem{theorem}{Theorem}[section]

\newtheorem*{main}{Main Theorem}
\newtheorem*{problem}{Main Problem}
\newtheorem*{prob}{Reduced Problem}

\newtheorem*{lem}{Lemma}
\newtheorem{proposition}{Proposition}[section]

\theoremstyle{definition}
\newtheorem{definition}{Definition}[section]
\theoremstyle{remark}

\newtheorem{remark}{Remark}[section]
\numberwithin{equation}{section}

\begin{document}
\title{Subsonic flows  past a profile with a vortex line at the trailing edge }

\author
{ Jun Chen \qquad Zhouping Xin \qquad Aibin Zang }
\address{J. Chen, Center of Applied Mathematics\\
         Yichun  University \\
        Yichun, Jiangxi 336000, P.R. China}
\email{chenjun@jxycu.edu.cn}
\address{Z. Xin, The Institute of Mathematical Sciences, The Chinese University of Hong Kong,\\
	Shatin, Hong Kong}
\email{zpxin@ims.cuhk.edu.hk}
\address{A. Zang, Center of Applied Mathematics\\
	Yichun  University \\
	Yichun, Jiangxi 336000, P.R. China}
\email{abzang@jxycu.edu.cn}
\date{}
%

\begin{abstract}
We established the existence, uniqueness and stability of  subsonic flows past an airfoil with a vortex line at the trailing edge. Such a flow pattern is governed by the two dimensional steady compressible Euler equations. The vortex line attached to the trailing edge  is a contact discontinuity for the Euler system and is treated as a free boundary. The problem is formulated and solved by using the implicit function theorem. The main difficulties are due to the fitting of the vortex line with the profile at the trailing edge and the possible subtle instability of the vortex line at the far field. Suitable choices of the weights and elaborate barrier functions are found to deal with such difficulties.
\end{abstract}

\maketitle

\section{Introduction}

A century ago, Prandtl developed the celebrated lifting line theory in \cite{Prandtl}, where he depicted the picture of a subsonic flow past an airfoil with a vortex line attached to the trailing edge (see Figure.\ref{fig:domainx}).
Although his theory has been  widely applied in the industry of aircraft designs, there is no mathematical justification of such solutions yet. The purpose of this paper is to provide a rigorous proof of the existence and uniqueness of the solutions to the airfoil problem with vortex lines.

In 1950's, Bers \cite{Bers,Bers2}, Finn and Gilbarg \cite{fg,fg2} studied airfoil problems for the potential flows. Since potential flows are rotation-free, the vortex lines do not appear in such type of flows. In order to study vortex lines, the full compressible Euler equations are needed (see equations \eqref{Euler1}--\eqref{Euler3}). In the theory of conservation laws, vortex lines belong to a class of contact discontinuities.

There are many studies related to steady subsonic flows in various  physically important situations, such as flows past a solid body and in nozzles, see \cite{Bers,Bers2,cdxx,cx,CCF,CCF2,CDX,CF,cj,dxx,fg,fg2,lxy,xx,xx2,xyy} and the references therein. However the literatures on steady compressible subsonic flows with contact discontinuities are limited, even for the case  when   airfoils  do not appear in the flows, with several notable exceptions: 
the contact discontinuities in Mach configurations are studied in \cite{Chensx} by using  the theory for elliptic equations with discontinuous data developed  in \cite{Liyy}; the contact  discontinuities in subsonic flows in nozzles have been analyzed in \cite{Bae} by making use  of the theory in \cite{Liyy} and  in \cite{BP} with a Helmholtz decomposition; while the authors in \cite{CHWX} employed  the theory of compensated compactness to obtain the contact discontinuities with large vorticity in arbitrary infinite long nozzles. We also refer to \cite{CKY,CKY2} for studies on contact discontinuities in transonic flows, \cite{WY}  in supersonic flows and \cite{cdxx,cx,CCF,CCF2,CDX,CF,cj,dxx,fg,fg2,lxy,xx,xx2,xyy} for other related  problems which involve subsonic Euler flows. However, all these works concerning contact discontinuities [1,2,11,15,25] do not involve the presence of  airfoils, which causes some new difficulties in the study of contact discontinuities.

In this paper, we will give a rigorous analysis for Prandtl's problem on subsonic flows past a finite thin airfoil with a vortex line.  The key points lie in understanding  the fitting of the trailing edge with the vortex line and the asymptotic behavior of the vortex line at the downstream.   The problem will be formulated and proved as the global structural stability (or instability) of a straight contact  line under the perturbation of finite thin profiles. More precisely, we take the straight contact line as the horizontal axis and the vertical line at the leading edge of the airfoil as the entrance of the flow (see Figure 1). By prescribing piecewise smooth data  corresponding background contact discontinuity on the entrance, and assuming  that the airfoil is suitably thin with the trailing edge being a cusp point (where the upper and lower boundaries of the airfoil meet at zero angle), we will find the unique subsonic solution with a vortex line attached to the trailing edge of the airfoil, which is H\"{o}lder continuous up to  the edges of the airfoil. Furthermore, the vortex line is shown to fit smoothly ($C^{1,\alpha}$) with the trailing edge of the profile and grows at most sublinearly at infinity. This implies in particular the structural stability of the background contact discontinuity. It is noted that the assumption of the trailing edge being a cusp is crucial here to guarantee  the smooth fitting of the vortex line at the trailing edge and the continuity of the flow up to the trailing edge.

Note that since the vortex line is part of the solution and an unknown, so this is a free boundary value problem for the full Euler system with slip boundary condition on the airfoil and suitable boundary conditions on the entrance. There are several new difficulties for treating such a problem. Since the full compressible Euler equations are a coupled elliptic-hyperbolic system for subsonic flows, the regularities of the unknowns  become crucial issues in analyzing such flows, even the formulation of boundary data becomes a subtle issue. Such difficulties are more pronounced for our problem here due to the presence of the airfoil and the vortex line. Indeed, there are some essential new difficulties in analyzing the subsonic flow past an airfoil with a vortex line at the trailing edge. The first one is about the smoothness of the fitting of the vortex line with the profile at the trailing edge. The second  one is that the flow is defined on the unbounded domain and the possible weak stability of the vortex line at far field. The other difficulty is related to the application of implicit function theorem as the framework of solving this problem.

The fitting of the trailing edge with the vortex line raises the issue about  corner regularities,   causing major difficulties in studying steady flows involving subsonic regions for the steady full Euler system due to the hyperbolic modes for subsonic flows ([5-10, 14, 21, 28]). Thus, handling the  regularities at the edges becomes an important issue. It should be noted that ``corner regularities'' are also key issues in the studying of transonic shocks in general curved nozzles (due to the intersection of the shock curve with the nozzles walls) as shown in \cite{xyy}. Yet, in this case, an Euler-Lagrangian type coordinates transformation can be introduced to decompose effectively the hyperbolic modes from the elliptic modes in the subsonic region after the shock and this reduces the Euler system into a nonlocal elliptic system for the pressure and the flow angle, so that one can obtain $C^\alpha$-regularity of the physical states in the subsonic region after the shock uniformly up to the corner (the intersection of the shock with the nozzle wall), see \cite{CCF,lxy}. These $C^\alpha$-regularities are the key to design some iteration schemes to solve the transonic shock problem in a generic 2-dimensional nozzles in \cite{CCF,lxy}. Motivated by such studies, we will look for solutions to the Prandtl's problem on subsonic flows past a thin finite profile with a vortex line, which is  $C^\alpha$-regular uniformly up to the profile including edges. To derive the $C^{\alpha} $ uniform regularity, we will decompose the Euler system for subsonic flows into hyperbolic system for the specific entropy and the  Bernoulli function and an elliptic system for the flow angle and reciprocal of the horizontal momentum. Under the crucial assumption that the trailing edge of the profile is a cusp, we can obtain the uniform $C^{\alpha}$ estimate up to the edges by constructing proper and  elaborate barrier functions. Note that these crucial estimates involve also the second difficulty that the flow domain is unbounded and the background contact discontinuity may not be asymptotically stable in the far field. The matching of the decay of physical states and the possible derivation of the vortex line from the background contact discontinuity at far field has to be carefully checked. Indeed, we will show that the vortex line may grow sublinearly away from the background contact discontinuity, but their unit normals approach each other in $L^\infty$-norm at far fields. These will be achieved by weighted H\"{o}lder estimates.

The other difficulty comes from the application of the implicit function theorem which we use as the framework to solve the problem of subsonic flows past an airfoil with a vortex line. Note that in many studies of transonic shocks [7-10,21,28], since the upstream supersonic flow can be obtained in advance, and the downstream subsonic flow and the position of the transonic shock have to be solved simultaneously by some elaborate iteration schemes with the position of the shocks as the solvability condition for the subsonic flow, thus the framework of either contraction mapping theorem or Schauder fixed point theorem can be used to find the desired transonic shock solutions.  However, for the problem of subsonic flows past an airfoil with a vortex line at the trailing edge, the flows on both sides of the airfoil and the vortex line are unknown and may grow asymptotically. This and the less regularity of the flows at the edges make it difficult to design an iteration scheme to apply the framework of the fixed point theorem as in the transonic shock problem. On the other hand, in suitable Euler-Lagrangian coordinate systems (see \eqref{def-coord}-\eqref{def-coord2} in \S3), both the upper and lower boundary of the airfoil and the vortex line become a stream line, thus one may try to use the approach in [1] to study our problem. Unfortunately, this does not work since the piecewise smooth subsonic solution do not form a weak solution across the segment of the particle line representing the airfoil. Thus, in this paper, we will use the implicit function theorem as the framework for solving the problem of subsonic flows past a thin airfoil with a vortex line attached at the trailing edge. This will require a suitable choice of weighted H\"{o}lder space, a proper design of a map $\TT$ (see \S5), and detailed analysis, especially for the isomorphism of the differential of the operator $\TT$.

It is also interesting to compare the problem studied in this paper with the transonic wedge problems in [8]. For the transonic wedge problem in [8], the regularity at the corner and the asymptotic behavior of the shock at far field are closely related through the oblique condition on the shock and the shock is asymptotically stable. However, for the problem of subsonic flows past an airfoil with a vortex line attached at the trailing edge, the upper and lower domains are separated by the airfoil and the vortex line. Since the conditions on the airfoil and the vortex line are different, which causes not only less regularity at the edges, but also gives rise to different phenomena from the transonic wedge problem in [8]. In fact, here the lower order regularities at edges can be localized so that they do not affect the asymptotic behavior of the vortex line at the far field, and contact discontinuities are weakly stable in the sense that the perturbed vortex line may grow at most sublinearly away from the background contact discontinuity and their unit normals converge to each other in $L^\infty$-norm at far field (see Remark \ref{r2.2}). This reveals significant differences between the structural stability of shocks and contact discontinuities.

The rest of the paper is organized as follows. In \S 2, we set up  the problem  and  state the main result. In \S 3, we  reformulate the the problem by introducing the Euler-Lagrange type coordinate transformations  and reducing the Euler equations into an elliptic system of two equations on upper and lower domains respectively.  \S 4 is about linearizations of the elliptic systems and related elliptic estimates. In \S 5, we define a map $\TT$ and show the properties of the differential of $\TT$. By using the implicit function theorem, we can solve an equation defined by $\TT$ and locate the contact discontinuity.

\section{Statement of the problem and the main result} \label{sec-setup}

Consider the following  two-dimensional steady
Euler equations:
\begin{align}
& \div (\rho \uu)=0,\label{Euler1}\\
&\div \left(\rho{\uu\otimes\uu}\right)
+\nabla p=0,\label{Euler2}\\
&\div \left(\rho\uu\left(E+\frac{p}{\rho}\right)\right)=0,
\label{Euler3}
\end{align}
where $\grad$ is the gradient in $\xx=(x_1,x_2)\in\R^2$, $\uu=(u_1,
u_2)^\top$ is the velocity field, $\rho$ is the density, $p$ is the pressure, and
$$
E=\frac{1}{2}|\uu|^2+\frac{p}{(\gamma-1)\rho}
$$
is the energy with adiabatic exponent $\gamma>1$. The sonic speed of
the flow is
$$
c=\sqrt{\gamma p/\rho}
$$
and the Mach number is
$$
M=\frac{|\uu|}{c}.
$$
For a subsonic flow, the Mach number $M <1$ at any point of the flow.

{\bf Conditions on contact discontinuity curves.} Suppose that a domain $\DD$  in $\R^2$ is divided by a $C^1$ curve $\CC$ into subdomains $\DD^+$ and $\DD^-$. Assume that  $U=(\rho, \uu, p)^\top$ is a piecewise $C^1$ solution of Euler equations \eqref{Euler1}--\eqref{Euler3} in each domain $\DD^+$ and $\DD^-$ and is continuous up to the boundary $\CC$ in each subdomain. We denote the restriction of $U$ on $\DD^+\cup \CC$ by $U^+$ and on $\DD^-\cup \CC$ by $U^-$. If $U$ is a weak solution for the Euler equations in the whole domain   $\DD$,   using integration by parts gives rise to the following so called Rankine-Hugoniot conditions on the curve $\CC$:
\begin{align}
&  [\rho \uu]\cdot \nn =0,\label{RH1}\\
&\left[\rho{\uu\otimes\uu}+pI\right]
 \nn  =\mathbf{0},\label{RH2}\\
&\left[\rho\uu (E+\frac{p}{\rho})\right]\cdot \nn=0,
\label{RH3}
\end{align}
where $\nn$ is the unit normal vector on $\CC$ and the bracket $[\ ]$ denotes the jump of the dependent variable from one subdomain to the other, i.e., for any smooth function $f: \R^4 \to \R$, $[f(U)]= f(U^+)- f(U^-)$.
Condition \eqref{RH2}, taken dot product with $\nn$ and with  a unit tangential vector $\gtt$ respectively, leads to
 \begin{align}
 &  [\rho (\uu \cdot \nn)^2 +p]  =0,\label{RH4}\\
 &[\rho (\uu \cdot \nn) (\uu \cdot \gtt)]  =0,\label{RH5}.
 \end{align}
When the flow moves across $\CC$, i.e., $\uu \cdot \nn \neq 0$ on $\CC$, with the entropy condition,  the curve $\CC$ is called a shock; if the flow moves along both sides of $\CC$ so that  $\uu \cdot \nn \equiv 0$ on $\CC$, then $\CC$ is said to be a contact discontinuity or a characteristic discontinuity. In the latter case, condition \eqref{RH4} with
 \begin{align}
& \uu^{\pm} \cdot \nn = 0\label{RH6}
\end{align}
 gives
\begin{align}
& p^+ = p^-\label{RH7}
\end{align}
along the contact discontinuity $\CC$. It is obvious that conditions \eqref{RH6} and \eqref{RH7} together ensure
the Rankine-Hugoniot conditions \eqref{RH1}--\eqref{RH4} and will be used later as the conditions for a contact discontinuity.

\begin{figure}
	\centering
	\includegraphics[width=0.75\textwidth]{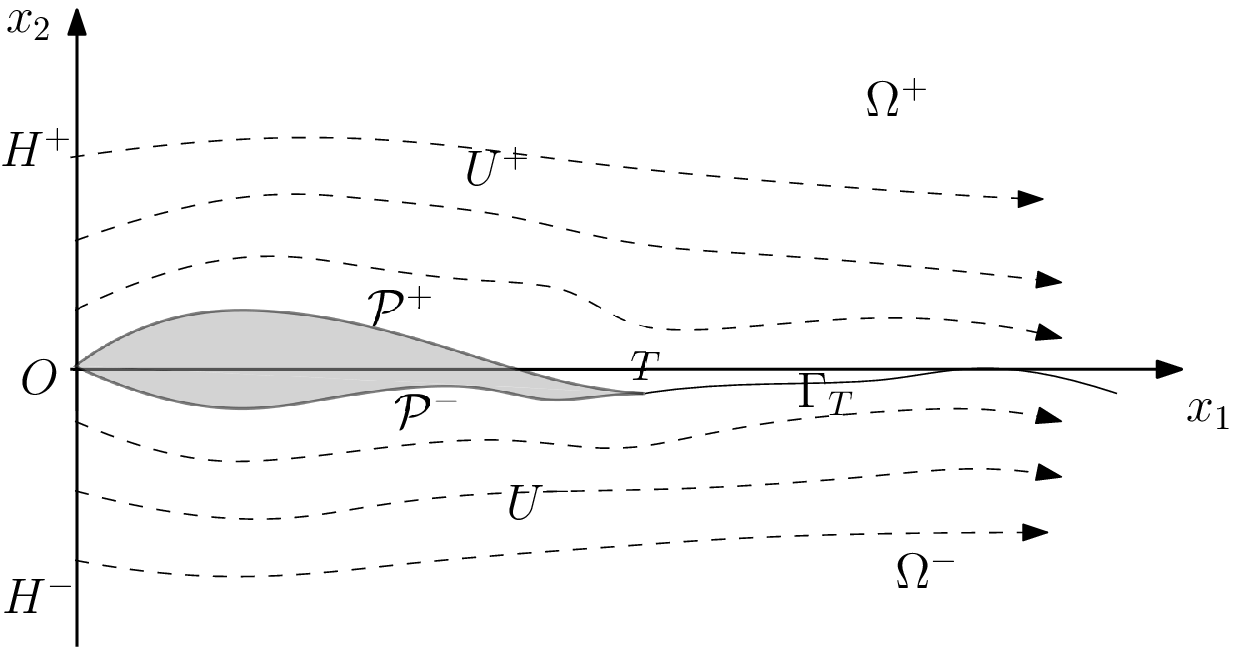}
	\protect\caption{Domain in $\xx$-coordinates.\label{fig:domainx}}
\end{figure}

The profile is bounded above by $\PP^+$ and below by $\PP^-$,  while   $\PP^+$ and  $\PP^-$ meet at the leading edge, fixed as the origin $O$,  and trailing edge $T$ (see Figure. \ref{fig:domainx}).

A contact discontinuity curve $\GC_T$ emanates from  $T$, and stretches  to infinity.  Denote $H^+,H^-$ as the positive and negative $x_2$-axis, respectively, and
 \begin{align*}
\GC^{\pm}:=   \PP^{\pm} \cup \GC_T .
 \end{align*}

Let $\Om^{\pm} $  be the open domains bounded by   the curves $H^{\pm}$ and $\GC^{\pm}$, respectively.
\begin{definition}
	A triple $(U^+, U^-,\Gamma_T)$ is called a {\bf subsonic profile solution} to the problem of a steady subsonic flow past a file with a vortex line at the trailing edge, if  the following conditions are satisfied:
	\begin{enumerate}
        \item $\Gamma_T$ is $C^1$-curve connecting $T$ to infinity;
		\item $U^{\pm}$  are $C^1$ solutions to Euler equations \eqref{Euler1}--\eqref{Euler3} in $\Om^{\pm}$, continuous up to $\GC^{\pm}$, respectively, possibly except at the edge points $O$ and $T$;
        \item $U^{\pm}$  satisfy  the contact discontinuity conditions \eqref{RH4} and  \eqref{RH5} on   $\GC_T$,   the slip condition  \eqref{RH6} on $\PP^{\pm}$, and suitable boundary conditions on $H^{\pm}$ and far field respectively;
        \item  $M^{\pm}(\xx) <1$  for all $\xx \in \Om^{\pm}$ .
 	\end{enumerate}
\end{definition}

Since the full Euler system is coupled elliptic-hyperbolic for subsonic flows, thus it is a subtle problem to formulate the suitable boundary conditions \cite{xx2}. To study the well-posedness of the problem for a subsonic profile solution, we need to formulate appropriate boundary conditions at both upstream and the far field.

Let the background contact discontinuity be given by some uniform states $(U^+_0, U^-_0)$ with a straight line as the corresponding contact continuity curve. By choosing a proper coordinate system, we may assume that the background contact discontinuity line is the horizontal axis. Then it follows from the definition of a contact discontinuity that
\begin{align}
U^\pm_0 &= (\rho^\pm_0, q^\pm_0, 0, p_0)^\top \label{U0+}
\end{align}
which are assumed to satisfy
\begin{align}\label{M0}
q^\pm_0 >0, \quad M^\pm_0<1,
\end{align}
where the last condition means that both $U^+_0$ and $U^-_0$ are subsonic. Take the origin, $O$, at the leading edge of the profile and the vertical axis as the upstream entrance of the flow, denoted as $H^\pm$, see Figure 1. At the upstream entrance $H^\pm$, we prescribe boundary data for the entropy function $A$, the Bernoulli quantity $B$, and the horizontal mass flux distribution $\mm \cdot \boldsymbol{\nu}= m_1$, where
\begin{align}\label{def-ABm}
A:= \frac{p}{\rho^{\gamma}},&& B:= \frac{1}{2}|\uu |^2 + \frac{\gc p }{(\gc-1)\rho} ,&& \mm = (m_1,m_2):=\rho \uu,
\end{align}
and $\boldsymbol{\nu}= (1,0)$  is the unit normal on the $H^{\pm}$. Precisely, we set
\begin{align} \label{con-bdHpm}
(A^{\pm},B^{\pm},m_1^{\pm}) = (A^{\pm}_0,B^{\pm}_0,m_{10}^{\pm}) \quad  \text{ on } H^{\pm}
\end{align}
with
\begin{align*}
 (A^{\pm},B^{\pm},m_1^{\pm}) &=    \left(\frac{p^{\pm}}{(\rho^{\pm})^{\gc}},  \frac{1}{2}|\uu^{\pm} |^2 + \frac{\gc p^{\pm} }{(\gc-1)\rho^{\pm}} , \rho^{\pm} u^{\pm}_1 \right)\\
  (A^{\pm}_0,B^{\pm}_0,m_{10}^{\pm})   &=    \left(\frac{p_0}{(\rho_0^{\pm})^{\gc}},  \frac{1}{2}(q_0^{\pm})^2 + \frac{\gc p_0}{(\gc-1)\rho_0^{\pm}} , \rho_0^{\pm} q^{\pm}_0 \right).
 \end{align*}

At the downstream, we require that the unit normal of $\Gamma_T$ approaches (0,1) as $x_1 \rightarrow +\infty$, and at the far field, the corresponding flow converges asymptotically to $U^\pm_0$ on $\Omega^\pm$ respectively.  More precisely, the unit normal of $\Gamma_T \rightarrow (0,1)$ as $x_1\rightarrow +\infty$,
\begin{align}
U^\pm(x) \rightarrow U^\pm_0 \quad \text{as} \quad x\in\Omega^\pm \quad \text{and} \quad |x|\rightarrow +\infty.\label{upmx}
\end{align}

Obviously, $(U^+_0, U^-_0)$ is a trivial profile solution for a degenerate profile where $\PP^+=\PP^-$ are horizontal segments and the corresponding connected contact discontinuity line is also horizontal.

The main aim of this paper is to   find a profile solution  close to the background contact discontinuity for a thin  airfoil.  We define $\PP^{\pm}$ by the graphs of functions $\ve \gz^{\pm}$ on $[0,1]$, where $\ve$ represents the thickness of the profile and $ \gz^{\pm}$ are given functions with
 \begin{align}
\label{con-zeta1} & \gz^- (t) \le \gz^+ (t), \  0\le t\le 1,\\
 &\gz^+(0) =\gz^-(0) = 0,  \\
&\gz^+(1) =\gz^-(1) = h_0, \\
\label{con-zeta4}& (\gz^+)'(1) =(\gz^-)'(1) = k^*,
 \end{align}
where $\gz^{\pm}$ are fixed, $h_0$ and $k^*$ are constants and so the trailing edge is $T=(1, \ve h_0)$, a cusp point where the slopes of both tangents of $\PP^+$ and  $\PP^-$ are $\ve k^*$. 

Suppose that the contact discontinuity curve  $\GC_T$ can be defined by the graph of  $g_T(t), t\ge 1$, which is continuously connected to the profile at the trailing edge, i.e.
\[
g_T ( 1) =\ve h_0.
\]

Define the boundary functions $g^{\pm}$  piecewisely as follows:
\begin{equation} \label{def-gc}
g^{\pm}(t) = \begin{cases}
\ve \gz^{\pm}(t) &\text{ if } \quad 0\le t\le 1 ,\\
g_T(t) & \text{ if } \quad t>1.
\end{cases}
\end{equation}

\begin{problem}
	Given  $U^\pm_0$ by \eqref{U0+} satisfying \eqref{M0} and a profile bounded by $\PP^\pm$, find a subsonic profile solution$(U^+, U^-,\Gamma_T)$ satisfying the upper stream boundary conditions \eqref{con-bdHpm} and the downstream far field conditions \eqref{upmx}.
\end{problem}

To state the precise results in this paper,  we need to define some weighted H\"older spaces and corresponding norms.  Let $\DD$ be a 2-dimensional domain and $P$ be a given set in $\R^2$. For any $x,x'\in\DD$, define
\begin{align*}
	&\gd_\xx := \min\left\{  \mbox{dist}(\xx,P), 1\right\},
	& &\gd_{\xx,\xx'} :=\min\{\gd_\xx, \gd_{\xx'}\},\\
	&\GD_\xx :=|\xx|+1,
	& &\GD_{\xx, \xx'}:= \min\{|\xx|+1, |\xx'|+1\}.
\end{align*}
Let $\ga \in (0,1)$, $\sigma, \tau \in \mathbb{R}$,    $k$
a nonnegative integer, $\kk = (k_1, k_2)$   an integer-valued
vector with $k_1, k_2 \ge 0$, $|\kk|=k_1 +k_2$, and $D^{\kk}=
\po_{x_1}^{k_1}\po_{x_2}^{k_2}$. For a function $f$ defined on $\DD$, we set
\begin{align}
\label{def-normsemi}
&[ f ]_{k,0;(\tau);\DD}^{(\sigma;P)}
:= \sup_{ \begin{array}{c}
	\xx\in \DD\\
	|\kk|=k
	\end{array}}\gd_{\xx}^{\max\{k+\sigma,0\}}\GD_\xx^{\tau+k}|D^\kk f(\xx)|, \\
\label{def-norm} & {[ f ]}_{k,\ga;(\tau);\DD}^{(\sigma;P)}
:= \sup_{
	\begin{array}{c}
	\xx, \xx'\in \DD\\
	\xx \ne \xx', |\kk|=k
	\end{array}}
\gd_{\xx,\xx'}^{\max\{k+\ga+\sigma,0\}}
\GD_{\xx,\xx'}^{\tau+k+\ga}\frac{|D^\kk f(\xx)-D^\kk f(\xx')|}{|\xx-\xx'|^\ga},
\\
& \|f\|_{k,\ga;(\tau); \DD}^{(\sigma;P)}
:= \sum_{i=0}^k {[f]}_{i,0 ;(\tau);\DD}^{(\sigma;P)}
+ {[f]}_{k,\ga;(\tau);\DD}^{(\sigma;P)}.
\label{def-normvector}
\end{align}
Define the weighted H\"older spaces as:
\begin{align}
C^{k,\ga;(\tau)}_{(\sigma;P)}(\DD)&:= \{ f:  \|f\|_{k,\ga;(\tau); \DD}^{(\sigma;P)} < \infty
\}.
\label{def-space}
\end{align}
For a vector-valued function $\ff=(f_1, f_2, \cdots,f_n )^\top$, we
define
\[
\|\ff\|_{k,\ga;(\tau);\DD}^{(\sigma;P)}
:=\sum_{i=1}^n  \|f_i\|_{k,\ga;(\tau);\DD}^{(\sigma;P)}.
\]

If $f$ is a function of one variable defined on an open set  $\LL \subset \R $ and $P$ is a set of some boundary or interior points of $\LL$,
the H\"older norms $\|f\|^{(\sigma;P)}_{k,\ga;(\tau);\LL}$ can be defined similarly. Furthermore, if there is no need to consider decay or growth rate at far field, we will drop the index  $\tau$,  written as $\|f\|^{(\sigma;P)}_{k,\ga;\DD}$. In this case, the factors  $\GD_\xx^{\tau+k}$ and $ \GD^{\tau+k +\ga}_{\xx,\xx'}$ in the definitions \eqref{def-normsemi} and \eqref{def-norm}, respectively, do not appear. Similarly, when the index $(\gs;P)$ does not appear in the norm, it means that $\gd_{\xx}^{\max\{k+\sigma,0\}}$ and $ \gd_{\xx,\xx'}^{\max\{k+\ga+\sigma,0\}}$ in the definitions \eqref{def-normsemi} and \eqref{def-norm}, respectively, do not appear.

 Denote
\begin{align*}
&\EE:= \{O,T\}.
\end{align*}
Then the main results in this paper can be stated as follows.

\begin{main}
Given $\ga\in (0,\frac{1}{2}), \gb\in (0,1)$,  $\zeta^{\pm} \in C^{3,\ga}_{((-\ga-1);\{0,1\})}((0,1)) $ satisfying conditions \eqref{con-zeta1}--\eqref{con-zeta4}, there exist positive constants $\ve_0$ and $C$, depending on $U_0^{\pm},\zeta^{\pm},\ga,\gb$, so that for any $\ve \in [0,\ve_0]$, there exists a subsonic profile solution $(U^+, U^-, \Gamma_T)$ with the related   contact discontinuity curve $\GC_T$ defined by the graph of $g_T$ to the {\bf Main Problem}, satisfying
\begin{align}
\nonumber & g_T(1)= \ve h_0, \quad g_T'(1) = \ve k^*,\\
  &\| U^{\pm} - U^{\pm}_0\|_{2,\ga;(\gb); \Om^{\pm}}^{(-\ga;\EE)}+ \| g_T\|_{3,\ga;( \gb-1); (1,\infty)}^{(-\ga-1 ;\{1\})}  \le C\ve. \label{est-Ugc}
\end{align}	
Such a solution $(U^+, U^-,\Gamma_T)$ to the {\bf Main Problem} is unique in the class defined by  \eqref{est-Ugc}.
\end{main}

\begin{remark}\label{r2.1} Some explanations  on the the assumptions and results in the main theorem are given as follows
\begin{enumerate}
\item For the thin airfoil,  the assumption $\zeta^{\pm} \in C^{3,\ga}_{((-\ga-1);\{0,1\})}((0,1)) $ implies that $\zeta^{\pm}$ are $C^{3,\ga}$ in the open interval $(0,1)$,  moreover, $\zeta^{\pm}$ are also $C^{1,\alpha}$ up to the corner points $O$ and $T$. However, the second derivatives of  $\zeta^{\pm}$ may be discontinuous at corner points  $O$ and  $T$.
\item The superscript $(-\alpha;\EE)$ of the term $ \| U^{\pm} - U^{\pm}_0\|_{2,\ga;(\gb); \Om^{\pm}}^{(-\ga;\EE)}$ in \eqref{est-Ugc}  implies that $ U^{\pm}$ have lower regularity  at corner points $O$ and  $T$.  In particular, $ U^{\pm}$ are  small perturbations from the background solutions $U^{\pm}_0$ up to the edges in the $C^{\alpha}-$norm. The index $\gb$ in the subscript implies that $ U^{\pm}(x)$ converge to $U^{\pm}_0$ as $x\in\Omega^\pm$ and $|x|\rightarrow\infty$ with the rate $|\xx|^{-\gb}$. 
\item The bound on $\| g_T\|_{3,\ga;( \gb-1); (1,\infty)}^{(-\ga-1 ;\{1\})} $ with superscript $(-1-\ga ;\{1\})$ means that the vortex curve $\Gamma_T$ is $C^{1+\alpha}$ continuous at the trailing edge $T$ so that the airfoil and the vortex line have a $C^{1+\alpha}$ continuous fitting at the trailing edge, while the subscript, index $\gb-1$, implies that the vortex curve $\Gamma_T$ may deviate from the background contact discontinuity line at the rate $x_1^{1-\gb}$ as $x_1 \to +\infty$, but the unit normal of $\Gamma_T$ converges to (0,1) at the rate $x^{-\beta}_1$ as $x_1\rightarrow+\infty$.
\end{enumerate}
\end{remark}

\begin{theorem}[Stability of vortex lines]\label{thm-stable}
Let $\ga,\gb, \zeta^{\pm}, \ve_0$ be the same as in the {\bf Main Theorem}. For any $\ve, \tilde{\ve} \in [0,\ve_0]$,  denote by  the corresponding subsonic profile solutions   $(U^{\pm}, \GC_T)$,  $(\tilde{U}^{\pm}, \tilde{\GC}_T)$ with the associated vortex lines $\GC_T$ and $\tilde{\GC}_T$ given as the graphs of $g_T$ and $\tilde{g}_T$,  respectively. Then
\begin{align}\label{est-stab}
 \| g_T - \tilde{g}_T\|_{3,\ga;( \gb-1); (1,\infty)}^{(-\ga-1 ;\{1\})}  \le C|\ve - \tilde{\ve} |.
\end{align}	
\end{theorem}

\begin{remark}\label{r2.2}
This theorem yields the uniform structual stability of the vortex line in the weighted H\"{o}lder space.  However, it should be pointed out that the uniform global stability estimate (2.22) does not yield the asymptotic stability in the super-norm of the vortex lines.  Instead, it can be checked that (2.22) allows sublinear divergence of the vortex lines in $L^\infty$-norm.  However, it can be checked easily that the stability estimate (2.22) implies that the unit normals $n(x)$ and $\tilde{n}(x)$ of the vortex line $\Gamma_T$ and $\Gamma_{\tilde{T}}$ converge to each other in $L^\infty$-norm as $x_1\rightarrow+\infty$.
\end{remark}

 \section{Reformulation of the {\bf Main Problem}} \label{sec-lagrange}

Let $(U^+,U^-,\Gamma_T)$ be a subsonic profile solution. Then the conservation of mass equation   \eqref{Euler1} gives  a
 stream function in each domain of  $\Om^+$ and  $\Om^-$.
  For convenience, we will focus on  $\Om^+$ and the notations  for the problem in $\Om^-$ will be replacing $+$ signs with $-$ signs.

  The stream function $\psi^+$ is uniquely determined in $\Om^+$ by
 \begin{align}
&\psi^+_{x_1}= -\rho^+ u^+_2, \quad \psi^+_{x_2} = \rho^+ u_1^+,\label{stream}
 \end{align}
 with
 \begin{align*}
\psi^+(0,0) =0.
\end{align*}

 Define the following coordinate
 transformation of Euler-Lagrange type
 \begin{equation}\label{def-coord}
 \left\{
 \begin{array}{lll}
 y_1=x_1,\\
 y_2= \psi^+(x_1, x_2),
 \end{array}
 \right.
 \end{equation}
 so that the streamlines in $\xx$-coordinates are mapped to horizontal lines in $ \yy$-coordinates. Correspondingly, domains $\Om^+$ and $\Om^-$ are mapped into the first quadrant $\QQ^+$ and the  fourth quadrant $\QQ^-$ respectively, and  the boundaries $\GC^+$ and $\GC^-$ both become positive  $y_1$-axis.

 \begin{remark}
 Let $g^+$ be the boundary function defined in \eqref{def-gc}. The coordinate transformation \eqref{def-coord} is equivalent to
 \begin{equation}\label{def-coord2}
 \left\{
 \begin{array}{lll}
 y_1=x_1,\\
 y_2= \int_{g^+(x_1)}^{x_2} \rho^+ u_1^+(x_1, s) \dx s.
 \end{array}
 \right.
 \end{equation}
 \end{remark}

 Furthermore, this transformation is globally invertable as long as $u^+_1>0$.

 In the new coordinates $\yy=(y_1,y_2)$, the unknown variables become $U^+(\yy):=U^+(\xx(y_{1},y_{2}))$ and the Euler system,  \eqref{Euler1}--\eqref{Euler3},  becomes
 \begin{align}\label{eqn-euler1}
 \left(\frac{1}{\rho^+ u_1^+} \right)_{y_1}
 -\left(\frac{u_2^+}{u_1^+} \right)_{y_2}&= 0,\\
 \left( u_1^+ + \frac{p^+}{\rho^+ u_1^+}\right)_{y_1}
 -\left( \frac{p^+ u^+_2}{u^+_1}\right)_{y_2}&= 0, \label{eqn-euler2}\\
 (u^+_2)_{y_1} + p^+_{y_2}&= 0,\label{eqn-euler3} \\
 B^+_{y_1}&=
 0.\label{eqn-euler4}
 \end{align}

 Equations \eqref{eqn-euler4} and (2.14) imply  the Bernoulli's law:
 \begin{equation}\label{eqn-bernoulli}
 \frac{1}{2}|\uu^+|^2 + \frac{\gc p^+}{(\gc-1)\rho^+} = B^+_0.
 \end{equation}

Equations \eqref{eqn-euler1}--\eqref{eqn-euler4} also give
 \[
 (\gc \ln \rho^+ -\ln p^+)_{y_1} =0,
 \]
which, combined with (2.14), leads to
 \begin{equation}\label{eqn-rho-p}
 p^+= A^+_0 (\rho^+)^\gc.
 \end{equation}

 Define a new variable in terms of the flow angle and the reciprocal of the horizontal momentum
 \begin{equation}\label{eqn-phi-deri}
 \vv^+ =( v_1^+, v_2^+ ) := \left( \frac{u^+_2}{u^+_1} ,   \frac{1}{\rho^+
 	u^+_1}\right).
 \end{equation}
 Then \eqref{eqn-euler1} becomes
  \begin{equation}\label{eqn-v1v2}
 (v^+_2)_{y_1} -  (v^+_1)_{y_2} = 0.
  \end{equation}
It follows from \eqref{eqn-phi-deri} and \eqref{eqn-rho-p} that the Bernoulli's law can be written as
 \begin{equation}\label{eqn-rho}
 \frac{(v^+_1)^2 + 1 }{2 (v^+_2)^2} + \frac{\gc}{\gc
 	-1}A^+_0(\rho^+)^{\gc+1} = B^+_0 (\rho^+)^2.
 \end{equation}

Since the flow is subsonic, so \eqref{eqn-bernoulli} and (3.9) imply
 \begin{equation}\label{con-subsonic}
 (\rho^+)^{\gc-1} > \frac{2(\gamma-1)B^+_0}{\gamma(\gamma+1)A^+_0}.
 \end{equation}

 Solving \eqref{eqn-rho} for  $\rho^+$ as  a function of $(A_0^+,B_0^+,
 \vv^+)$ gives  two solutions, corresponding to the subsonic and supersonic solutions respectively.    Condition \eqref{con-subsonic} singles out  the unique value for  $\rho^+$ in the subsonic regime, and we denote it as
 $$\rho^+ = \rho(A_0^+, B_0^+, \vv^+). $$
Thus, the solution $U^+$ is expressed as a vector-valued function of $ \vv^+$:
 \begin{equation}\label{eqn-express-U}
 U^+= (\rho^+, u_1^+,u_2^+, p^+)^\top= \left(\rho^+, \frac{1}{\rho^+ v^+_2}, \frac{v^+_1}{\rho^+ v^+_2} , A_0^+(\rho^+)^\gc\right)^\top.
 \end{equation}
 Denote
 \begin{equation}\label{def-N}
 \begin{cases}
 u^+_2:= N^+_1( \vv^+) =N_1(A^+_0, B^+_0, \vv^+)\\
 p^+:=  N_2^+ (\vv^+)= N_2(A^+_0, B^+_0, \vv^+).
 \end{cases}
 \end{equation}
 Then \eqref{eqn-euler3} becomes
 \begin{equation}\label{eqn-N+}
 (N^+_i( \vv^+))_{y_i} =0,
 \end{equation}
where the Einstein's convention for the summation has been used.

The slip condition \eqref{RH6} on $\GC^+$ implies that
 \begin{equation}\label{Con-bd-phi+-wall}
 v_1^+ (y_1, 0) = (g^+)'(y_1),  \quad y_1 \ge 0.
 \end{equation}
Condition (2.14) on  $H^+$ gives rise to
  \begin{equation}\label{Con-bd-H+phi+}
v_2^+ (0, y_2) = \frac{1}{m_{10}^+}.
  \end{equation}

Similarly, in $\QQ^-$, one has that
 \begin{equation}\label{eqn-N-}
(N^-_i( \vv^-))_{y_i}  =0, (v^-_2)_{y_1} -  (v^-_1)_{y_2} = 0.
\end{equation}
with boundary condition
 \begin{equation}\label{Con-bd-phi--wall}
 v_1^- (y_1, 0) = (g^-)'(y_1),  \quad v_2^- (0, y_2) = \frac{1}{m_{10}^-}.
\end{equation}

On the contact discontinuity line  $\GC_T$, \eqref{RH7} becomes
 \begin{equation}\label{con-RHpN2}
N^+_2( \vv^+(y_1,0))=N^-_2( \vv^-(y_1,0)),  \quad y_1>1.
\end{equation}

Therefore, the {\bf Main Problem} is reformulated as follows:
\begin{prob}
Suppose  that $(U^+_0, U^-_0)$ and $\PP^{\pm}$ are given  in the \textbf{Main Problem}.  Find the functions   $g_T$ for the contact discontinuity curve $\GC_T$, and  $\vv^{\pm}$ such that the followings are satisfied:
\begin{enumerate}
	\item $\vv^{+}$ and $\vv^{-}$ are $C^2$ solutions to equations   \eqref{eqn-v1v2} and \eqref{eqn-N+} in $\QQ^+$ and the equations \eqref{eqn-N-} in $\QQ^-$ respectively;
	\item Boundary conditions  \eqref{Con-bd-phi+-wall},  \eqref{Con-bd-H+phi+}, \eqref{Con-bd-phi--wall}  and \eqref{con-RHpN2}
   are satisfied;
	\item At far field, the asymptotic behavior, $  \vv^{\pm} (\yy) \to (0,\frac{1}{m_{10}^{\pm}})$, $\yy\in\QQ^\pm$ and $ |\yy| \to +\infty$, is satisfied.
\end{enumerate}

\end{prob}

\section{Solution to A Fixed Boundary Value Problem in $\QQ^+$} \label{sec-solvephi}

Note that the {\bf Main Problem} in the previous section is a free boundary problem since the contact discontinuity curve $\Gamma_T$ is unknown. As we mentioned in the introduction, this free boundary problem will be solved by using the implicit function theorem. To this end, it turns out that the key step is to solve two fixed boundary value problems in $\QQ^\pm$ respectively for a given $\Gamma_T$. Since the solvability for the boundary value problem on $\QQ^-$ is similar to that for $\QQ^+$, we will focus on the solvability of the \eqref{eqn-v1v2} and \eqref{eqn-N+} with the boundary condition (3.17) and (3.18) for the given vortex line $\Gamma_T$ in $\QQ^+$. this nonlinear problem will be solved by the following three steps:
\begin{enumerate}
\item  Linearize the nonlinear fixed boundary value problem and reformulate the linearized problem as a mixed boundary value problem for a second order linear elliptic equation;
\item Solve    the reformulated linear problem;
\item Construct a map $\FF$ based on linearization problem and show the existence of a fixed point of $\FF$, which yield the solution the nonlinear problem, \eqref{eqn-v1v2},\eqref{eqn-N+} subject to the boundary condition (3.17) and (3.18) in $\QQ^+$.
 \end{enumerate}

{\bf Step 1. Linearization and reformulation.} Given $\vv^+ = \vv_0^+ + \gd \vv$, find $ \vvb^+ = \vv_0^+ + \gd \vvb$ by solving equations
 \begin{align}\label{eqn-v1bv2b}
  (\gd \vb_2)_{y_1} -  (\gd \vb_1)_{y_2} &= 0, \\
( N^+_i (\vv^+) )_{v_j^+} ( \gd \vb_{j})_{y_i} &= 0,\label{eqn-lin-gdvb}
 \end{align}
with boundary conditions
\begin{align}\label{con-bdlinvb1}
\gd \vb_1   (y_1, 0) &= (g^+)'(y_1), \quad y_1 \ge 0;  \\
 \gd \vb_2 (0, y_2) &= 0, \quad y_2 >0. \label{con-bdlinvb2}
\end{align}

Let $q^+= \sqrt{(u^+_1)^2 + (u^+_2)^2}$.
Direct calculations show
that
\begin{align}
& (N^+_1)_{v^+_1}  =\frac{u^+_1 ((c^+)^2 -( u^+_1)^2)}{(c^+)^2 -(q^+)^2},\\
\label{exp-N1phi2}& (N^+_1)_{v^+_2}  = (N^+_2)_{v^+_1} = -\frac{(c^+)^2\rho^+ u^+_1u^+_2}{(c^+)^2 -(q^+)^2},\\
& (N^+_2)_{v^+_2}  = \frac{(c^+)^2 (\rho^+)^2(q^+)^2u^+_1}{(c^+)^2 -(q^+)^2}.
\end{align}
Hence, for a subsonic flow with $u^+_1 >0$, it holds that
\begin{align*}
(N^+_1)_{v^+_1} >0,\quad (N^+_2)_{v^+_2} >0.
\end{align*}
It follows from \eqref{eqn-v1bv2b} and \eqref{eqn-lin-gdvb} that
 \begin{align}
 ( a^{ij} (\vv^+)   ( \gd \vb_1)_{y_j})_{y_i} &= 0,\label{eqn-lin-gdvb2order}
\end{align}
where
\begin{align}\label{def-aijv}
a^{11}= \frac{(N^+_1)_{v^+_1}}{(N^+_2)_{v^+_2}}, \quad a^{12}= \frac{2(N^+_1)_{v^+_2}}{(N^+_2)_{v^+_2}},\quad  a^{21}= 0, \quad a^{22}=1.
\end{align}

Then for a subsonic flow with $u^+_1>0$,
\begin{align}
a^{11} - \left(\frac{a^{12}}{2}\right)^2 & = \frac{(N^+_1)_{v^+_1}(N^+_2)_{v^+_2} -
	(N^+_1)_{v^+_2}^2}{(N^+_2)_{v^+_2}^2}  \nonumber\\
&= \frac{((c^+)^2 -(q^+)^2) ( u^+_1)^2}{(c^+)^2(\rho^+)^2(q^+)^4} >0,\label{elliptic}
\end{align}
which shows that the equation \eqref{eqn-lin-gdvb2order} is elliptic.

To recover equation \eqref{eqn-lin-gdvb} from \eqref{eqn-lin-gdvb2order}, we prescribe \eqref{eqn-lin-gdvb} on $H^+$ as a boundary condition. Since $\gd \vb_2 = 0$ on $H^+$, this condition can be written as
 \begin{align}\label{con-oblique-v1}
   a^{11} (\vv^+)  ( \gd \vb_1)_{y_1} +  a^{12} (\vv^+)  ( \gd \vb_1)_{y_2} &= 0.
 \end{align}

Subsonicity condition \eqref{M0} for $U^+_0$ and \eqref{elliptic} imply the uniform ellipticity of \eqref{eqn-lin-gdvb2order}, provided that $\vv^+$ is a small perturbation from $\vv_0^+$. More precisely, there  exist constants $\gl, \gs >0$, depending only on $U^+_0$, such that when
\[
\|\gd \vv\|_{L^{\infty}(\QQ^+)} \le \gs,
\]
for any $\boldsymbol{\xi} \in \R^2$, it holds that
\begin{align}\label{unifellip}
\gl |\boldsymbol{\xi}|^2 \le a^{ij}(  \vv^+) \xi_i\xi_j \le \frac{1}{\gl} |\boldsymbol{\xi}|^2.
\end{align}
Set
\begin{align}\label{def-GGgs}
\GG^\gs &:= \{  v:\| v \|_{2,\ga;(\gb); \QQ^+}^{(-\ga;\EE)} \le \gs \}, \\
\label{def-GGgs*} \GG^\gs_*& := \{  v: v(0,1)= \ve k^*, \| v \|_{2,\ga;(\gb); \QQ^+}^{(-\ga;\EE)} \le \gs \},
\end{align}
where
$$\EE = \{O,T\},\quad    T=(1,0).$$	
Given $\gd \vv \in \GG^\gs_* \X \GG^\gs$, we will solve equations \eqref{eqn-v1bv2b}, \eqref{eqn-lin-gdvb2order} with boundary conditions \eqref{con-bdlinvb1}, \eqref{con-bdlinvb2}, \eqref{con-oblique-v1} by $\gd \vvb \in C^{2,\ga;(\gb)}_{(-\ga;\EE)}(\QQ^+) \X C^{2,\ga;(\gb)}_{(-\ga;\EE)}(\QQ^+)$.

For the purpose of applications later, instead of solving the above problem, we consider a more general setup. Set
\[
L^+:= [0,\infty)\X\{0\}.
\]
We will solve the following mixed boundary value problem ({\bf Problem MB})
\begin{align*}
\begin{cases}
a^{ij}   v_{y_iy_j}+ b^i  v_{y_i }= f  & \text{ in }  \QQ^+,\\
v = g  &\text{ on }      L^+ ,\\
\mu^i\po_{y_i}v = g_1  &\text{ on }  H^+.
\end{cases}	
\end{align*}
{\bf Step 2. Solve the linear problem.}  Results about  existence and uniqueness of solutions to mixed boundary problems can be found in Lieberman's paper \cite{Lieberman}. However, since here the domain is unbounded and boundary data may have low regularity, so Lieberman's theorems can not be applied directly. Thus,  we may need to carry out estimates on some truncated domains. To this end, we will use the following notations. For constants $\delta$ and $R$ such that $0<\delta<1$, $R>\delta+1$, set
\begin{align}\label{def-domain}
\begin{cases}
 \QQ^R  = B_R(O)\cap \QQ^+, &\QQ^R_\gd = \QQ^R \backslash  \overline{B_\gd(O) \cup B_\gd (T)},\\   S^R:= \po B_R(O)  \cap \QQ^+,
& H^R  = \{0\} \X (0,R),\\
 L^R   = [0,R]\X  \{0\},
\end{cases}
\end{align}
where $B_R(X) $ denotes the disc centered at point $X$ with radius $R$.

\begin{figure}
	\centering
	\includegraphics[width=0.5\textwidth]{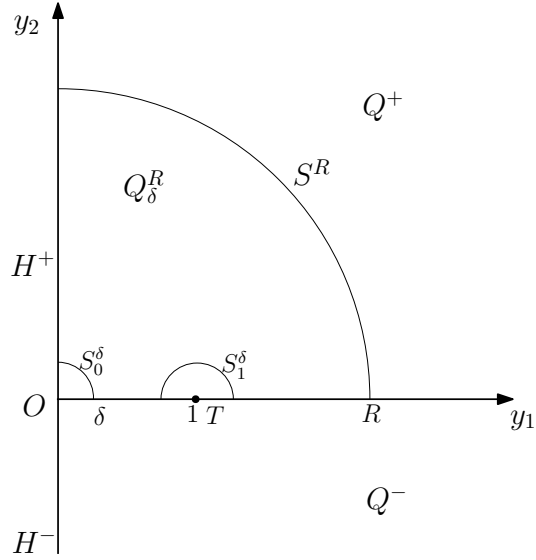}
	\protect\caption{Truncated domain in $\yy$-coordinates.\label{fig:domainy}}
\end{figure}

In order to prescribe a boundary condition on $S^R \cup L^R_\gd$, one can  extend the domain of  $(g^+)'$ from $\R^+$ to $\QQ^+$ in the following way.

Let  $K(t)$ be a smooth mollifier satisfying
\begin{align*}
K(t)\geq 0, \quad \mathrm{supp}\, K \subset [-1,1] ,  \quad \int_{\mathbb{R}}K(t)\dx t=1,
\end{align*}
and   $\eta \in C_c^{\infty}(\mathbb{R})$ be a cut-off function   with the following properties
\begin{align}\label{def-cutoff}
\begin{cases}
0 \le \eta(t)\le 1, & t\in \R ; \\
\eta(t)=1,& t\in [-1,1];\\
\eta(t)=0& t\in (-\infty, -2]\cup [2,\infty).
\end{cases}
\end{align}
Define for $(y_1,y_2)\in\overline{ \QQ^+}$
 \begin{align}\label{def-g}
g (\yy) &:=\eta(y_2)\left( \int_{0}^\infty (g^+)' (y_1 + t  y_2)	K(t) \mathrm{d}\, t+\int_{-\infty}^{0} (g^+)' (y_1 - t  y_2)	K(t) \mathrm{d}\, t\right),
 \end{align}
Then we prescribe  the following Dirichlet condition :
 \begin{align}\label{con-bd-vb1g}
 \gd \vb_1 &= g
 \end{align}
on $S^R \cup L^R$.

\begin{remark}
	The following estimate holds for $g$ defined in \eqref{def-g}:
 \begin{align}\label{est-g}
 \|g \|_{2,\ga;(\gb); \QQ^+}^{(-\ga;\EE)}  & \le C \| g^+\|_{3,\ga;( \gb-1); (1,\infty)}^{(-1-\ga ;\{0,1\})},
\end{align}	
where $C$  depends only on the choices of the kernel $K$ and the cut-off function $\eta$. We refer to the proof of Lemma 2.3 in \cite{GH} about the details  for the proof of  \eqref{est-g}.
\end{remark}

The following result for the truncated problem  in $ \QQ^R$ holds:
\begin{proposition}\label{prop-estimate-vf}
	Suppose that $a^{ij}, \mu^i ,g \in C^{2,\ga;(\gb)}_{(-\ga;\EE)}(\QQ^+)$, $b^i, g_1\in C^{1,\ga;(1+\gb)}_{(1-\ga;\EE)}(\QQ^+)$, $ f\in C^{0,\ga;( 2+\gb)}_{(2-\ga;\EE)}(\QQ^+)$ and
	\begin{align}\label{con-aij}
	\|a^{ij} -\gd^{ij} \|_{2,\ga;(\gb); \QQ^+}^{(-\ga;\EE)} \le \tau,   &&	\|\mu^i -\gd^{1i} \|_{2,\ga;(\gb); \QQ^+}^{(-\ga;\EE)} + \|b^{i}   \|_{0,\ga;(2-\gb); \QQ^+}^{(1-\ga;\EE)}  \le \tau,
	\end{align}
where $ i,j=1,2$,  $\gd^{ij} = \begin{cases}
		1 & \text{ if } \quad i=j\\
		0 &\text{ if } \quad  i\neq j
		\end{cases}.$

	Then  there exists a suitably small constant $\tau>0$, such that the following truncated mixed boundary value problem ({\bf Problem MR})
	\begin{align*}
	\begin{cases}
	a^{ij}   v^{R}_{y_iy_j}+ b^i  v^{R}_{y_i }= f  & \text{ in }  \QQ^R,\\
v^{R} = g  &\text{ on }    S^R \cup L^R,\\
\mu^i\po_{y_i}v^{R} = g_1  &\text{ on }  H^R
	\end{cases}	
	\end{align*}
	has a unique solution in $C^{2,\ga} (\QQ^R) \cap C^0 (\overline{\QQ^R})$ for any fixed positive constant   $R >4 $. Furthermore, the following estimate holds:
	\begin{align}\label{est-vf-r/2}
	\|v^{R} \|_{2,\ga;(\gb); \QQ^{\frac{R}{2}} }^{(-\ga;\EE)} \le C \left( \|f\|_{0,\ga;(2+\gb); \QQ^+}^{(2-\ga;\EE)} + \|g \|_{2,\ga;(\gb); \QQ^+}^{(-\ga;\EE)}  + \|g_1 \|_{1,\ga;(1+\gb); \QQ^+}^{(1-\ga;\EE)}\right),
	\end{align}
where the constant $C$  depends only on $\tau$, but is independent of $R$.
\end{proposition}

\begin{proof}[Proof of Proposition \ref{prop-estimate-vf}]

By Theorem 1 in Lieberman's paper \cite{Lieberman}, there is a unique solution $v^R$ in $C^2(\QQ^R) \cap C^{\ga'} (\overline{\QQ^R}) $ for some $\ga' >0$ to {\bf Problem MR}. Next, we will derive some estimates independent of $R$ so that one can let $R \to \infty$ to obtain the solution to {\bf Problem MB}.
	
 Set
\[
\kappa:=   \|f\|_{0,\ga;(2+\gb); \QQ^+}^{(2-\ga;\EE)} + \|g \|_{2,\ga;(\gb); \QQ^+}^{(-\ga;\EE)}  + \|g_1 \|_{1,\ga;(1+\gb); \QQ^+}^{(1-\ga;\EE)}.
\]
{\bf Decay estimates.} { Since $b^i, f, g$ and $g_1$ may have lower regularities  at  $\EE$}, it is convenient to use    the polar coordinates    $(r, \theta)$  centered at $O$  and $(r^*, \theta^*)$ centered at $T$  respectively  to construct barrier functions. Define
 \begin{align}
\vf_1(\yy) & :=r^{-\gb}\sin (\gb_1\theta + \theta_0), \label{def-vf1}\\
\vf(\yy)&: = C\kappa  \vf_1(\yy),\label{def-vf}
\end{align}
where $\gb < \gb_1 <\alpha, \theta_0>0, \gb_1 \pi/2 +\theta_0 < \pi/2$, here and hereafter, all  $C$'s and $C_i$'s, $i=1,2,\cdots$, denote generic positive constants, which may depend on $U_0^{\pm},\zeta^{\pm},\ga,\gb$.

Direct calculations show that
 \begin{align}
\GD \vf_1(\yy) &= (\gb^2- \gb_1^2)r^{-\gb-2}\sin (\gb_1\theta + \theta_0), \label{barrier1}\\
( \vf_1)_{y_1}(0,y_2 )& = - r^{-1}\po_{\theta}\vf_1|_{\theta=\frac{\pi}{2}} =  - \gb_1  r^{-\gb-1}\cos (\gb_1\pi/2 + \theta_0).\label{ND}
 \end{align}
 Let $\tau$ be suitably small and choose proper $C$. Then it follows from \eqref{con-aij} and \eqref{def-vf1}-\eqref{ND} that
 \begin{align*}
 \begin{cases}
 a^{ij}   \vf_{y_iy_j}+ b^i \vf_{y_i }\le f & \text{ in }  \QQ^R,\\
\vf  \ge g  & \text{ on }    S^R \cup L^R,\\
\mu^i\po_{y_i}\vf  \le g_1  & \text{ on }  H^R.
\end{cases}
\end{align*}
Consequently, the standard maximum principle for elliptic equations shows that $\varphi-v^{R}$ cannot achieve a minimum in $\QQ^R$.  Since $\mu^i\po_{y_i}(\vf-v^{R}) < 0$ on $H^R$ and $(\mu^1,\mu^2)\cdot(1,0)>0$ for suitably small $\tau$, so $\varphi-v^{R}$ cannot achieve a minimum on $H^R$ neither. It thus follows that
\begin{align}\label{est-vf-max}
v^{R} \le \vf (\yy) .
\end{align}
Similarly, one can check that $-\varphi$ is a strict sub-solution to {\bf Problem MR}, thus it holds that
\begin{align}\label{est-vf-min}
v^{R} \ge -\vf (\yy) .
\end{align}
It follows from \eqref{est-vf-max}, \eqref{est-vf-min} and \eqref{def-vf} that
\begin{align}\label{est-vf-bdd}
|v^{R} (\yy) |\le \vf (\yy) \le C \kappa r^{-\gb}.
\end{align}

{\bf $C^0$ estimate.} It follows from \eqref{est-vf-bdd} that
$$|v^{R} (\yy)| \le C \kappa,  \quad \text{ on }  S^2.$$
In $\QQ^2$, one can define a barrier function as
\[
\vf_2  := C \kappa  (r^{\gb_2}\sin (\gb_1\theta + \theta_0) + {(r^*)}^{\gb_2}\sin (\gb_1\theta^{*} + \theta^{*}_0) )
\]
with $\gb_1$ and $\theta_0$ are given as in \eqref{def-vf1}, $0<\gb_{2}<\gb_{1}, \theta^{*}_0>0$, and $\beta_{1}\pi+\theta^{*}_0<\pi.$ Then this enables one to get
\[
|v^{R} (\yy)| \le \vf_2 (\yy) \le C\kappa \quad \text{ for }   \yy \in \QQ^2.
\]
The estimate above and estimate \eqref{est-vf-bdd} yield the following $C^0$ estimate for $v^{R}$ in $\QQ^R$:
 \begin{align}\label{est-vfRgdC0}
 |v^{R} (\yy)| \le  C\kappa \min(1, r^{-\gb}).
 \end{align}

{\bf Weighted  $C^{2,\ga}$ estimates on $\QQ_{\frac{1}{4}}^{\frac{R}{2}} $.} Choosing proper scalings and applying Schauder estimates will lead to the desired estimate \eqref{est-vf-r/2}. Indeed, for any $\yy_0$ in $\QQ_{\frac{1}{4}}^{\frac{R}{2}} $, let $d^*= \frac{1}{8} \dist (\yy_0, \EE)$, $\DD_1 = B_{d^*}(\yy_0) \cap \QQ^+ $, $\DD_2 = B_{2d^*}(\yy_0) \cap \QQ^+ $ and $\CC = \po \DD_2 \backslash \po \QQ^+ $.   The Schauder interior estimate  of Theorem 6.2 in \cite{gt} on the domain $\DD_2$  gives rise to
\begin{align}\label{est-vRgd1}
\|v^{R} \|_{2,\ga; \DD_2 }^{(0;\CC)} \le C (\|v^{R} \|_{0,0; \DD_2 } + \|f \|^{(2;\CC)} _{0,\ga; \DD_2 }),
\end{align}
here $\|v^{R}\|_{0,0;\DD_2}=\sup_{x\in \DD_2}|v^{R}(x)|$ and $C$ depends  only on $\alpha, \lambda$.

Near   $S^{R}\cup L^{R}$, by the boundary estimate of Lemma 6.4 in \cite{gt} on the domain $\DD_{2}$,  one can get 
\begin{align}\label{est-vRgd2}
\|v^{R} \|_{2,\ga; \DD_2 }^{(0;\CC)} \le C (\|v^{R} \|_{0,0; \DD_2 } +\|g\|_{2,\ga; \DD_2 }^{(0;\CC)}+ \|f \|^{(2;\CC)} _{0,\ga; \DD_2 })
\end{align}
where  $C$ does not depend on $R$.

Near $H^{R}$,  applying the Theorem 6.26 in \cite{gt} on the domain $\DD_{2}$ yields
\begin{align}\label{est-vRgd3}
\|v^{R} \|_{2,\ga; \DD_1 }^{(0;\CC)} \le C (\|v^{R} \|_{0,0; \DD_2 } +\|g_{1}\|_{1,\ga; \DD_2 }^{(1;\CC)}+ \|f \|^{(2;\CC)} _{0,\ga; \DD_2 })
\end{align}
here  $C$ is independent  of $R$.

{ Since $\xx,\yy\in\DD_{2}$, $\frac{1}{4}\le\delta_{\xx},\delta_{\yy}\le 1$, then
\begin{equation}\label{f-norm}
\begin{aligned}
 \|f \|^{(2;\CC)} _{0,\ga; \DD_2 }&=\sup_{\xx\in \DD_{2}}d^{2}_{\xx}|f(\xx)|+\sup_{\xx,\yy\in\DD_{2}; \xx\ne \yy}d_{\xx,\yy}^{2+\alpha}\frac{|f(\xx)-f(\yy)|}{|\xx-\yy|^{\alpha}}\\
 &\le C |\yy_{0}|^{-\gb}\left(\sup_{\xx\in \DD_{2}}d^{2}_{\xx}(d^{*})^{\gb}|f(\xx)|+\sup_{\xx,\yy\in\DD_{2}; \xx\ne \yy}d_{\xx,\yy}^{2+\alpha}(d^{*})^{\gb}\frac{|f(\xx)-f(\yy)|}{|\xx-\yy|^{\alpha}}\right)\\
 &\le C |\yy_{0}|^{-\gb}\|f\|_{0,\ga;(2+\gb);\QQ^{+}}^{2-\alpha;\EE}\le C\kappa |\yy_{0}|^{-\beta}
\end{aligned}
\end{equation}
where $d_{\xx}=\dist\{\xx,\CC\}$ and $d_{\xx,\yy}=\min\{d_{\xx}, d_{\yy}\}$. 
Similarly, we infer that 
\begin{equation}\label{g-g1-norm}
\|g\|_{2,\ga; \DD_2 }^{(0;\CC)}\le  C\kappa |\yy_{0}|^{-\beta}; \|g_{1}\|_{1,\ga; \DD_2 }^{(1;\CC)}\le C\kappa |\yy_{0}|^{-\beta}.
\end{equation}}

For any $y_0 \in \QQ_{\frac{1}{4}}^{\frac{R}{2}}$,  from \eqref{est-vRgd1} to \eqref{est-vRgd2}, we have
\begin{equation}\label{est-y0}
\begin{aligned}
&\sum_{|\kk|=0}^{2}(d^{*})^{|\kk|}|D^{\kk}v^{R}(\yy_{0})|+\sup_{
	\begin{array}{c}
	\yy\in \DD_{1},\yy \ne \yy_{0}\\
	 |\kk|=2
	\end{array}}{(d_{\yy,\yy_{0}})^{2+\alpha}}\frac{|D^{\kk}v^{R}(\yy)-D^{\kk}v^{R}(\yy_{0})|}{|\yy-\yy_{0}|^{\alpha}}\\
\le{}&  C (\|v^{R} \|_{0,0; \DD_2 } +\|g\|_{2,\ga; \DD_2 }^{(0;\CC)}+\|g_{1}\|_{1,\ga; \DD_2 }^{(1;\CC)}+ \|f \|^{(2;\CC)} _{0,\ga; \DD_2 }).
\end{aligned}
\end{equation}

Therefore,  estimates \eqref{est-vfRgdC0}, \eqref{est-y0}, \eqref{f-norm} and \eqref{g-g1-norm} yield
\begin{equation}\label{est-y02}
\begin{aligned}
\sum_{|\kk|=0}^{2}(d^{*})^{|\kk|}|D^{\kk}v^{R}(\yy_{0})|&+\sup_{
	\begin{array}{c}
	\yy\in \DD_{1},\yy \ne \yy_{0}\\
	 |\kk|=2
	\end{array}}{(d_{\yy,\yy_{0}})^{2+\alpha}}\frac{|D^{\kk}v^{R}(\yy)-D^{\kk}v^{R}(\yy_{0})|}{|\yy-\yy_{0}|^{\alpha}}\\
&\le  C\kappa |\yy_{0}|^{-\beta}.
\end{aligned}
\end{equation}

This leads to the following estimate:
\begin{align}\label{est-vR14}
\|v^{R} \|_{2,\ga;(\gb);\QQ_{\frac{1}{4}}^{\frac{R}{2}} }\le C \kappa.
\end{align}

{\bf Corner estimates near $O$ and $T$.} Consider domain $\QQ^{\frac{1}{2}}$ for the corner estimates near $O$.  We use the following barrier function 
\begin{align*}
\bar{\vf} (\yy):= C\kappa r^{\ga}\sin (\gb_3\theta + \theta_0) \quad (\ga<\gb_3 <1)
\end{align*}
for $\vb^R(\yy):= v^R(\yy) - v^R(0,0)$ and obtain
\begin{align}\label{est-vbRc0}
|\vb^R(\yy)|\le C \kappa r^{\ga},  \quad \forall \yy \in \QQ^{\frac{3}{4}}.
\end{align}

Using the Schauder estimates derived in the same way for \eqref{est-y0} and the same scaling argument as for \eqref{est-y02}, one gets that
\begin{align}\label{est-vbR12prime}
\sum_{|\kk|=0}^{2}(d^{*})^{|\kk|}|D^{\kk}\vb^{R}(\yy_{0})|+\sup_{	\begin{array}{c}
	\yy\in \DD_{3},\yy \ne \yy_{0}\\
	 |\kk|=2
	\end{array}}(\tilde{d}_{\yy,\yy_{0}})^{2+\alpha}\frac{|D^{\kk}\vb^{R}(\yy)-D^{\kk}\vb^{R}(\yy_{0})|}{|\yy-\yy_{0}|^{\alpha}} \le C  \kappa|\yy_0|^{\ga},
\end{align}
for any $\yy_0\in \QQ^{\frac{1}{2}} $, where  $\DD_3 = B_{\frac{|\yy_0|}{4}}(\yy_0), \tilde{d}_{\yy,\yy_{0}}=\min\{dist(\yy,\partial\DD_{3}), dist(\yy_{0},\partial\DD_{3})\}$. Thus, one can get the following corner estimate:
\begin{align}\label{est-vRcorner1}
\|\vb^R \|^{(-\ga;\{O\})}_{2,\ga; \QQ^{\frac{1}{2}} } \le C  \kappa.
\end{align}
Similarly, the following corner estimate near $T$ holds:
\begin{align}\label{est-vRcorner2}
\|\vb^R \|^{(-\ga;\{T\})}_{2,\ga; B_{\frac{1}{2}}(T)\cap \QQ^+ } \le C  \kappa.
\end{align}
Estimate \eqref{est-vf-r/2} follows from  estimates \eqref{est-vR14}, \eqref{est-vRcorner1} and  \eqref{est-vRcorner2}. Hence, the proof of Proposition \ref{prop-estimate-vf} is completed.

\end{proof}

By estimate \eqref{est-vf-r/2}, one can choose a   subsequence $\{ v^{R_i}  \}$ converging to $v$ in each space $C_{(-\ga';\EE)}^{2,\ga';(\gb)} (\QQ^{\frac{R_i}{2}} )$ for a fixed $\ga' \in (0,\ga)$, as $ R_i \to \infty$.  Then  \eqref{est-vf-r/2} implies
	\begin{align}\label{est-v}
	\|v \|_{2,\ga;(\gb); \QQ^+ }^{(-\ga;\EE)} \le C \left( \|f\|_{0,\ga;(2+\gb); \QQ^+}^{(2-\ga;\EE)} + \|g \|_{2,\ga;(\gb); \QQ^+}^{(-\ga;\EE)}  + \|g_1 \|_{1,\ga;(1+\gb); \QQ^+}^{(1-\ga;\EE)}\right).
	\end{align}
Obviously, $v$ is a solution to  the {\bf Problem MB}.

For the  uniqueness in space $C^{2,\ga;(\gb)}_{(-\ga;\EE)}(\QQ^+)$, it suffices to show that there is only a trivial solution in $C^{2,\ga;(\gb)}_{(-\ga;\EE)}(\QQ^+)$ to the {\bf Problem MB} with $f=g=g_1=0$. Define  a barrier function as
\begin{align}\label{def-vf4}
\vf_3 (\yy):=  c_0 \left(C_3 r^{-\gb'}\sin (\gb_1\theta + \theta_0) +  \eta(4r^* )  (r^*)^{-\gb}\sin (\gb_1\theta^{*} + \theta^{*}_0)  \right),
\end{align}
where $0 < \gb' < \gb$, { $\gb_1, \theta_0$ are given in the definitions \eqref{def-vf1} and $r^*,\theta^{*}, \theta^{*}_0 $  are the same as for the definition of $\vf_{2}$,  $\eta(t)$ is  defined in \eqref{def-cutoff}},   $c_0$ is a  positive constant which will approach to 0 later.

Since $\gb' < \gb$ implies that $v$ decays faster than $\vf_3$ as $|\yy| \to +\infty$, for large enough $R$, then
\[
|v(\yy)| \le \vf_3(\yy), \quad \text{ for }   \yy \text{ on }  S^R.
\]
As for the calculations of \eqref{barrier1} and \eqref{ND}, one can check that $\vf_3$ is  a  super-solution  of the following equation
\begin{align*}
	\begin{cases}
	a^{ij}   v_{y_iy_j}+ b^i  v_{y_i }= 0  & \text{ in }  \QQ^R,\\
v = 0 &\text{ on }    S^R \cup L^R,\\
\mu^i\po_{y_i}v = 0  &\text{ on }  H^R,
	\end{cases}	
	\end{align*}
provided that $\tau$ is suitable small.

As for \eqref{est-vf-bdd}, one shows that  $\vf_3 - v$ can achieve the minimun neither in $\QQ^R$ nor on $H^{R}.$  Note that $v$ is bounded and $\vf_3(\yy) \to +\infty$ as $\yy \to O$ or $T$.  Hence $\vf_3 - v$ can achieve minimum only on $S^R \cup (0,1)\X \{0\} \cup (1,R)\X \{0\}$ and the minimum is nonnegative. Therefore, we conclude that
\[
|v(\yy)| \le \vf_3(\yy), \quad \text{ for }   \yy \in \QQ^R.
\]
Thus, letting $R \to \infty$ yields that
\[
|v(\yy)| \le \vf_3(\yy)  , \quad \text{ for }   \yy \in \QQ^+.
\]
Let $c_{0}\to 0$, we conclude that $v\equiv 0$ in $\QQ^+$.
\begin{remark}
In fact, it follows from  Proposition \ref{prop-estimate-vf} that the solution is H\"{o}lder continuous  at the corners $O$ and $T$. Since the {\bf Problem MB} is linear, it is easy to see the uniqueness. The proof above shows that  the uniqueness holds even in the case   that the bounded solutions may be singular at the corners $O$ and $T$.  
\end{remark}

 The results discussed above  can be summarized in the following proposition:
\begin{proposition}\label{prop-estimate-gdvb1}
	Under the same assumptions as in Proposition \ref{prop-estimate-vf}, there exists a  solution $v$ to the {\bf Problem MB} satisfying the estimate \eqref{est-v}.  Furthermore, the solution is unique in $C^{2,\ga;(\gb)}_{(-\ga;\EE)}(\QQ^+)$.
\end{proposition}

We now apply Proposition \ref{prop-estimate-gdvb1}  to solve the boundary value problem \eqref{eqn-lin-gdvb2order}, \eqref{con-bdlinvb1}  and \eqref{con-oblique-v1}. Though the coefficients in \eqref{eqn-lin-gdvb2order}  and \eqref{def-aijv} may not satisfy the assumption \eqref{con-aij} directly, yet  by a suitable rescaling of $y_1$ (depending only on the background solution), one may reduce the problem \eqref{eqn-lin-gdvb2order}, \eqref{con-bdlinvb1}  and \eqref{con-oblique-v1}  to the following problem
\begin{align*}
\begin{cases}
\widetilde{a^{ij}} v_{y_{ij}} + \widetilde{b^i}v_{y_i}=0 & \text{in} \quad \QQ^+\\
v=g\equiv (g^+)' & \text{on}\quad \LL^+\\
\widetilde{\mu^i}\partial_{y_i}v=0 &\text{ on }  H^+
\end{cases}
\end{align*}
with
\begin{align*}
\widetilde{a^{11}}=\frac{a^{11}}{a^{11}(v^+_0)}, \quad \widetilde{a^{12}}=\frac{a^{12}}{\sqrt{a^{11}(v^+_0)}}, \quad \widetilde{a^{21}}=0, \quad \widetilde{a^{22}}=1\\
\widetilde{b_i}=\partial_{y_j} (\widetilde{a^{ji}}), \quad \widetilde{\mu^1}=\frac{a^{11}}{a^{11}(v^+_0)}, \quad \widetilde{\mu^2}=\frac{a^{12}}{\sqrt{a^{11}(v^+_0)}}.
\end{align*}

It can be checked easily that the coefficients above satisfy \eqref{con-aij}, so one can apply Proposition \ref{prop-estimate-gdvb1}  to conclude that there exists a unique solution $\gd \vb_1$ to the problem \eqref{eqn-lin-gdvb2order}, \eqref{con-bdlinvb1}  and \eqref{con-oblique-v1}, such that
\begin{align}\label{est-gdvb1}
	\|\gd \vb_1 \|_{2,\ga;(\gb); \QQ^+ }^{(-\ga;\EE)} \le C    \|g \|_{2,\ga;(\gb); \QQ^+}^{(-\ga;\EE)} \le C	\|(g^+)' \|_{2,\ga;(\gb); \R^+}^{(-\ga;\{0,1\})}.
\end{align}

Set
\begin{align}\label{exp-gdvb2}
\gd \vb_2 (\yy) = \int_0^{y_1} (\gd \vb_1)_{y_{2}}(t, y_2) \,dt, \quad \forall \, \yy \in \QQ^+.
\end{align}
Then the equation \eqref{eqn-v1bv2b} and the boundary condition \eqref{con-bdlinvb2} are satisfied. Next, we derive the super-norm estimate on $\gd\vb_2$.

{\bf Boundedness of $\gd \vb_2$.}
The right hand side of \eqref{exp-gdvb2} is integrable for each fixed $y_2 >0$, which implies that $\gd \vb_2 $ is bounded for each fixed $y_{2}>0$. However, the bounds are not uniform in $y_2$, due to less regularity of the solutions  at $O$ and $T$. Observe that for any given point $\yy^* \in \overline{\QQ^+}\backslash \EE$, and any smooth curve $\CC$ starting from $(0,1)$, ending at $\yy^*$, with parametrization $\yy(t), t\in [0,1]$, it holds that:
\begin{align}\label{eqn-vb2integral}
\gd \vb_2 (\yy^*) = \int_0^1 \grad (\gd \vb_2) (\yy(t))\cdot \yy'(t) \,\dx t.
\end{align}
It follows from \eqref{eqn-v1bv2b}, \eqref{eqn-lin-gdvb} and \eqref{est-gdvb1} that
 \begin{align}\label{est-vb2higher}
| \grad (\gd \vb_2 (\yy) ) |\le C | \grad (\gd \vb_1 (\yy) ) | \le C 	\|(g^+)' \|_{2,\ga;(\gb); \R^+}^{(-\ga;\{0,1\})}  |\yy|^{-1-\gb}.
 \end{align}
 Then for any $\yy$ away from $\EE$ (with distance greater than $\frac{1}{2}$), \eqref{eqn-vb2integral} leads to
 \begin{align}\label{est-gdvb2bd}
  | \gd \vb_2 (\yy)   |\le   C 	\|(g^+)' \|_{2,\ga;(\gb); \R^+}^{(-\ga;\{0,1\})} \quad \forall \, \yy \in \QQ^+, \ \dist(\yy,\EE) \ge \frac{1}{2}.
 \end{align}
For $\dist(\yy,\EE) < \frac{1}{2} $, \eqref{est-gdvb1} yields the following possible blowup rate for $\grad (\gd \vb_1 (\yy) )$:
 \begin{align*}
| \grad (\gd \vb_1 (\yy) ) |\le  C 	\|(g^+)' \|_{2,\ga;(\gb); \R^+}^{(-\ga;\{0,1\})}  \gd_{\yy}^{-1+\ga}.
\end{align*}
Hence,  the following  boundedness for  $\gd \vb_2 (\yy) $ holds:
\begin{align}\label{est-gdvb2log}
| \gd \vb_2 (\yy)   |\le   C 	\|(g^+)' \|_{2,\ga;(\gb); \R^+}^{(-\ga;\{0,1\})} , \quad \forall \, \yy \in \QQ^+, \ \dist(\yy,\EE) < \frac{1}{2}.
\end{align}

%
It follows that
 \begin{align}\label{est-phibar}
 \| \gd \vvb \|_{L^{\infty}(\QQ^{+})}  \le  C 	\|(g^+)' \|_{2,\ga;(\gb); \R^+}^{(-\ga;\{0,1\})} .
\end{align}

{\bf Decay estimates of $\gd \vb_2$.} To guarantee $\gd \vb_2 \in \GG^\gs$, besides estimates \eqref{est-phibar}, one needs to derive the decay of $\gd \vb_2 $ at far field. In the similar way as solving for $\gd \vb_1$ in {\bf Step 1}, one differentiates equation \eqref{eqn-v1bv2b} w.r.t. $y_1$, equation \eqref{eqn-lin-gdvb} w.r.t. $y_2$ to eliminate $\gd \vb_1$ and obtain the equation for $V:=\gd\vb_{2}$ as
\begin{align}
( \tilde{a}^{ij} (\vv^+)   V_{y_i})_{y_j} &= 0,\label{eqn-lin-V}
\end{align}
where
\begin{align*}
\tilde{a}^{11}=1 , \quad \tilde{a}^{12}= \frac{2(N^+_1)_{v^+_2}}{(N^+_1)_{v^+_1}},\quad  \tilde{a}^{21}= 0, \quad \tilde{a}^{22}=\frac{(N^+_2)_{v^+_2}}{(N^+_1)_{v^+_1}}.
\end{align*}
The boundary conditions can be prescribed as  follows:
\begin{align} \label{con-oblique-v2}
\begin{cases}
V=0  & \text{on } H^+,\\
V= \gd \vb_2  & \text{on }  (0,1)\X \{0\},\\
\tilde{a}^{12} (\vv^+) V_{y_1} +  \tilde{a}^{22} (\vv^+) V_{y_2}  = (g^+)'' & \text{on }   ( 1, \infty)\X \{0\}.
\end{cases}
\end{align}
Since  $\gd \vb_1$ and  $\gd \vb_2$ solve the equations \eqref{eqn-v1bv2b} and \eqref{eqn-lin-gdvb}, it is easy to see that $\gd\vb_{2}$  also solves the equations \eqref{eqn-lin-V} and \eqref{con-oblique-v2}. 
We  define the following barrier function, similar to $\vf_3$ in \eqref{def-vf4}:
 \begin{equation}
 \begin{aligned}\label{def-vf5}
 \vf_4 (\yy):=&{}  \ve  (C_4r^{\gb'}\sin (\gb_1\theta + \theta_0)  + \eta(4r )  r^{-\gb'}\sin (\gb_1\theta + \theta_0) \\
 &{} +  \eta(4r^* )  (r^*)^{-\gb'}\sin (\gb_1\theta^{*} + \theta^{*}_0)),
 \end{aligned}
 \end{equation}
 where $0<\gb'<\gb_1$, while $\theta_0,\theta^{*}_0, \gb_1$ satisfy the following conditions:
\begin{align} \label{con-angles}
 0<\gb<\gb_1<1, \theta_0 >\pi/2,\theta^{*}_0 >0, \gb_1 \pi/2 +\theta_0 < \pi,\gb_1 \pi +\theta^{*}_0 < \pi.
 \end{align}
 This enables one to show  the uniqueness of bounded solutions of  the equations \eqref{eqn-lin-V} and \eqref{con-oblique-v2}, which implies that the solution $V$ has to be $\gd \vb_2$.

On the other hand,  \eqref{est-vb2higher} and \eqref{est-gdvb2log} imply $\|\gd\vb_{2}\|_{2,\ga; (0,1)}^{(-\ga;\{0,1\})}\le \|(g^+)' \|_{2,\ga;(\gb); \R^+}^{(-\ga;\{0,1\})}$.  Similar to the estimates for \eqref{est-gdvb1}, using the barrier function $\vf$ defined by \eqref{def-vf}   with the conditions \eqref{con-angles} can yield a solution $V$ to \eqref{eqn-lin-V},  \eqref{con-oblique-v2}  with the following estimate
 \begin{align*}
 \|V \|_{2,\ga;(\gb); \QQ^+ }^{(-\ga;\EE)} \le   C	\|(g^+)' \|_{2,\ga;(\gb); \R^+}^{(-\ga;\{0,1\})}.
 \end{align*}
This together with \eqref{est-gdvb1}, implies the estimate
   \begin{align} \label{est-vbar}
  \|\gd \vvb \|_{2,\ga;(\gb); \QQ^+ }^{(-\ga;\EE)} \le   C	\|(g^+)' \|_{2,\ga;(\gb); \R^+}^{(-\ga;\{0,1\})}.
  \end{align}

{\bf Step 3. The fixed point of $\FF$.} Assume
\begin{align}\label{con-g+small}
  \|(g^+)' \|_{2,\ga;(\gb); \R^+}^{(-\ga; \{0,1\})} \le \tau,
\end{align}
where  $C \tau < \gs$ with $C$  given in \eqref{est-vbar}.  Then estimate \eqref{est-vbar} guarantees $\gd \vvb \in \GG_{*}^\gs \X \GG^\gs$.

Thus we define a map $\FF: \GG_{*}^\gs \X \GG^\gs \to \GG_{*}^\gs\X \GG^\gs $ by $\gd \vvb := \FF(\gd \vv)$  to solve the equations \eqref{eqn-v1bv2b}  and \eqref{eqn-lin-gdvb} with the boundary conditions \eqref{con-bdlinvb1}  and \eqref{con-bdlinvb2} for given  $\gd \vv$. Observe that $ \GG_{*}^\gs \X \GG^\gs $ is a compact convex set in $ C^{2,\ga';(\gb)}_{(-\ga ;\EE)}(\QQ^+)  \X C^{2,\ga';(\gb)}_{(-\ga;\EE)}(\QQ^+) $ for $0<\ga'<\ga$ and $\FF$ is continuous. Indeed, the continuity of $\FF$ can be proved by the argument for the estimate \eqref{est-v}. By the Schauder fixed point theorem, there exists a fixed point $\gd \vv^+ \in \GG_{*}^\gs \X \GG^\gs$ for $\FF$. Hence, $\vv^+ = \vv_0^+ + \gd \vv^+$ is a solution to equations \eqref{eqn-v1v2}, \eqref{eqn-N+} with boundary conditions \eqref{Con-bd-phi+-wall},  \eqref{Con-bd-H+phi+}.

 Now we show the uniqueness of the solution  for $  \vv^+ \in \GG_{*}^\gs \X \GG^\gs$.

 Suppose that  $\vv^+$ and $\vvb^+$ are both solutions to \eqref{eqn-v1v2} and \eqref{eqn-N+} with boundary conditions \eqref{Con-bd-phi+-wall}, \eqref{Con-bd-H+phi+}. Then $\gd \vv := \vvb^+ -\vv^+$ satisfies equation \eqref{eqn-v1v2} and the following equation
\begin{align*}
&\left(\left(\int_0^1 (N^+_i)_{v^+_j}(\vv^+ + t\gd \vv)\dx t \right) \gd v_j\right)_{y_i}  =0,
\end{align*}
with boundary conditions
\begin{align*}
\gd v_1 (y_1,0) =0, \quad \gd v_2 (0,y_2) =0.
\end{align*}
Then by the estimate \eqref{est-v} in the proof  of Proposition \ref{prop-estimate-gdvb1},  $\gd \vv \equiv 0$, provided  that $\gd \vv  \in \GG^{2\gs} \X\GG^{2\gs}$. In summary, we have the following result:
\begin{proposition}\label{prop-nonlinear}
For suitably small $\gs$ and $\tau$, if $g^+$ satisfies  \eqref{con-g+small}, then there exists a  solution $\vv^+$ to equations  \eqref{eqn-v1v2},  \eqref{eqn-N+} with boundary conditions \eqref{Con-bd-phi+-wall}, \eqref{Con-bd-H+phi+}, satisfying the following estimate:
	\begin{equation}
	\|\vv^+ - \vv^+_0  \|_{2,\ga;( \gb); \QQ^+}^{(-\ga;\EE)} \le C \|(g^+)' \|_{2,\ga;( \gb); \R^+}^{( -\ga;\{0,1\})}. \label{est-gdv+}
	\end{equation}
	The solution is unique for $ \vv^+ - \vv^+_0 \in \GG^\gs_*  \X\GG^\gs$.
\end{proposition}

\begin{remark} Similarly, for given suitably $g^-$, one can solve the fixed boundary value problem, \eqref{eqn-N-}-\eqref{Con-bd-phi--wall}, on $\QQ^-$ to obtain a unique solution $\vv^-$ satisfying \eqref{est-gdv+} with $\vv^+$, $\vv^+_0$, $g^+$  and $\QQ^{+}$ replaced by $\vv^-$, $\vv^-_0$, $g^-$  and $\QQ^{-}$respectively.
\end{remark}

\section{Construction and properties of the slip line} \label{sec:map}

With the solvability and the corresponding estimates of the fixed boundary value problems (with fixed $\Gamma_T$), Proposition 4.3 and Remark 4.3, we will be able to solve the free boundary value problem, the {\bf Reduced Problem}, by an elaborate use of the implicit function theorem in this section.

Define a  Banach space
\begin{align*}
\Sigma:= C^{2,\ga;(\gb)}_{(-\ga; \{1\})}((1,\infty)),
\end{align*}
equipped with norm
\begin{align*}
\|\cdot \|_\Sigma:= \| \cdot \|_{2,\ga;(\gb);(1,\infty)}^{(-\ga;\{1\})} .
\end{align*}
Set
\begin{align*}
V^{\ve_1}_*= \{ w    \in \GS: w(1)= \ve k^*, \|w\|_\GS < \ve_1 \}.
\end{align*}
For any $w \in V^{\ve_1}_*$, let
\begin{align}
g_T(t)& := \ve h_0 + \int_1 ^t w (s) \dx s , \quad t>1 \label{def-gcT}
\end{align}
and $g^{\pm}$  be  defined by \eqref{def-gc}. Choose suitably small $\ve_1$  so that \eqref{con-g+small} holds for $w\in V^{\ve_1}_* $ and $|\ve| \le \ve_1 $.
It then follows from Proposition 4.3 and Remark 4.3 that there exist $\vv^+ $ and $\vv^-$ which solve the fixed boundary value problem, \eqref{eqn-v1v2} and \eqref{eqn-N+}-\eqref{Con-bd-H+phi+} on $\QQ^{+}$ and \eqref{eqn-N-}-\eqref{Con-bd-phi--wall} on $\QQ^{-}$, respectively.  Set
\begin{align}\label{def-gdp}
   [\, p \,] &:= \left.\left(N_2^+(  \vv^+) - N_2^-(  \vv^-) \right) \right|_{(1,\infty)\X \{0\} },
\end{align}
and then define a map $\TT: (-\ve_1,\ve_1)\X V^{\ve_1}_* \to \GS$ by
\begin{align*}
\TT (\ve,  w) &:= [\, p \,].
\end{align*}

The condition \eqref{RH7} for the pressure, i.e. (3.21), can be written as the equation
\begin{align*}
\TT (\ve,  w)  = \mathbf{0},
\end{align*}
which will be solved by the implicit function theorem stated below for convenience.

\begin{lem}[Implicit function theorem]
	Let $X,Y,Z$ be Banach spaces over $\R$ and $V(x_0,y_0)$ is an open neighborhood of $(x_0,y_0)$ in $X \X Y$. Suppose that the mapping $F: V(x_0,y_0) \to Z$ satisfies:
	\begin{enumerate}
		\item $F(x_0,y_0)=0$;
		\item The partial Fr\'{e}chet derivative $D_y F$  exists on  $V(x_0,y_0)$ and  $D_y F(x_0,y_0): Y \to Z$  is bijective;
		\item $F$ and $D_y F$ are continuous at $(x_0,y_0)$.
	\end{enumerate}
Then there exist two positive constants $\ve_1$ and $\ve_2$, such that for any $x \in X$ satisfying $\|x -x_0\| \le \ve_1$, $F(x,y)=0$ has a unique solution  $y(x) \in Y$ with $\|y(x) -y_0\| \le \ve_2$.

Furthermore, if $F$ is continuous  in a neighborhood of  $(x_0,y_0)$,   then $y(\cdot)$ is also continuous in a neighborhood of  $x_0$; if  $F$  is a $C^1$-map on a neighborhood of  $(x_0,y_0)$, then $y(\cdot)$ is also a  $C^1$-map on a neighborhood of  $x_0$.

\end{lem}

The precise statement of this lemma is given by Theorem 4.B page 150 in \cite{ze}.

We now verify the conditions in the implicit function theorem.

Obviously, $\TT (0,\mathbf{0})= \mathbf{0}$.

{\bf Differentiability of $\TT$.} Consider the perturbation of  a given point  $(\ve,  w)$ as $(\tilde{\ve}, \tilde{ w}) = (\ve +\gd \ve,  w + \gd  w)$. The boundary data $g^{\pm}$ change to $\tilde{g}^{\pm} = g^{\pm} +\gd g^{\pm}$. It follows from the definition  \eqref{def-gcT} that
\begin{align*}
\gd g^{\pm} (t) &  = \begin{cases}
 \gd \ve \gz^{\pm}(t), & 0 \le t \le 1;\\
 \gd \ve h_0 + \int_1 ^t  \gd w (s)\, \dx s, & t>1.
\end{cases}
\end{align*}
We linearize equations \eqref{eqn-N+} and \eqref{eqn-N-} around $(\vv^+, \vv^-)$ respectively and solve the following linear problems:
\begin{align}\label{eqn-DT}
 \begin{cases}
\left((N^{\pm}_i)_{v^{\pm}_j}( \vv^{\pm})  \gd v^{\pm}_j\right)_{y_i}=0  &  \text{ in } \QQ^{\pm};\\
 \gd v_1^{\pm} =(\gd g^{\pm})' & \text{  on }  L^{+};\\
  \gd v_2^{\pm} =0 & \text{  on } H^{\pm}.
\end{cases}
\end{align}
By Proposition \ref{prop-estimate-gdvb1}, the above problems are uniquely solvable in $ C^{2,\ga;( \gb)}_{(-\ga;\EE)}(\QQ^{\pm})$, if $  \vv^{\pm} - \vv_0^{\pm} \in \GG^\gs_* \X \GG^\gs$

Then we linearize  the right hand side of \eqref{def-gdp} and set
\begin{align*}
\gd p  :=  \left. \left((N_2^+(\vv^+))_{\vv^+}\cdot  \gd \vv^+ - (N_2^-(  \vv^-))_{\vv^-}\cdot   \gd \vv^-   \right )\right|_{(1,\infty)\X \{0\} }.
\end{align*}
Noticing that $\gd p $ is linear in $\gd \vv^{\pm}$ and $\gd \vv^{\pm}$ is linear in $(\gd \ve, \gd  w)$, we define a linear map $D\TT(\ve, w)$ of $(\gd \ve, \gd  w)$ by
\begin{align*}
D\TT(\ve, w) (\gd \ve, \gd  w) :=   \gd p.
\end{align*}
Since  $w(1) = \tilde{w} (1) = \ve k^* $, we require that $\gd w (1) =0$. Set
\begin{align*}
\Sigma_0 := \{v \in \Sigma:  v(1) = 0\}.
\end{align*}
Then $D\TT(\ve, w)$ is a linear map from $(-\ve_1,\ve_1) \X \Sigma_0$ to $  \Sigma $.

Now we show that $D \TT (\ve, w)$ is the differential of $\TT$ at $ (\ve, w)$.

Let $\tilde{\vv}^+$ be the solution to \eqref{eqn-v1v2} \eqref{eqn-N+} with boundary data $\tilde{g}^+$ and $\Phi :=\tilde{\vv}^+ - \vv^+ - \gd \vv^+ $. Then $\Phi$ satisfies
\begin{align*}
  \left((N^+_i)_{v^+_j}(  \vv^+)   \Phi_j \right)_{y_i} &= F,
\end{align*}
with homogeneous boundary data,  where
\begin{align*}
 F=\left(\left((N^+_i)_{v^+_j}( \vv^+) - \int_0^1 (N^+_i)_{v^+_j}( \vv^+ + t (\tilde{\vv}^+ - \vv^+))\dx t \right) (\tilde{ v}_j^+ - v_j^+)\right)_{y_i}.
\end{align*}
It thus follows from the proof of Proposition \ref{prop-estimate-gdvb1} that
	\begin{align}\label{est-PhiF}
	\|\Phi \|_{2,\ga;( \gb); \QQ^+}^{(-\ga;\EE)} \le  C\|F\|_{0,\ga;(2+\gb); \QQ^+}^{(2-\ga ;\EE)} \le C 	\left(\|\tilde{\vv}^+ - \vv^+ \|_{2,\ga;( \gb); \QQ^+}^{(-\ga;\EE)}\right)^2.
	\end{align}
Since $\tilde{\vv}^+ - \vv^+ $ satisfies
\begin{align}\label{eqn-phit-phi}
\left(\left( \int_0^1 (N^+_i)_{v^+_j}( \vv^+ + t (\tilde{\vv}^+ - \vv^+))\,\dx t \right) (\tilde{ v}_j^+ - v_j^+ )\right)_{y_i}=0,
\end{align}
with boundary data $(\gd g^+)'$, so Proposition \ref{prop-estimate-gdvb1} yields the following estimate
	\begin{align}\label{est-phit-phi}
 \|\tilde{\vv}^+ - \vv^+ \|_{2,\ga;( \gb); \QQ^+}^{(-\ga;\EE)} \le C
 \|(\gd g^+)' \|_{2,\ga;(\gb); \R^+}^{(-\ga; \{0,1\})} \le C(|\gd \ve| + \|\gd  w\|_{\GS})
	\end{align}
Similar estimates hold in $\QQ^-$ with $H^-$. This and estimates \eqref{est-PhiF} and \eqref{est-phit-phi} lead to
	\begin{align}\label{est-phit-phi2}
	\|\tilde{\vv}^{\pm} - \vv^{\pm} - \gd \vv^{\pm}\|_{2,\ga;( \gb); \QQ^{\pm}}^{(-\ga;\EE)} \le
C(|\gd \ve| + \|\gd  w\|_{\GS})^2.
	\end{align}
Hence, with estimates \eqref{est-phit-phi} and \eqref{est-phit-phi2}, we have
\begin{align*}
&\left\| \TT(\tilde{\ve}, \tilde{ w}) - \TT (\ve,  w) - D\TT(\ve, w) (\gd \ve, \gd  w)\right\|_{\GS}\\
={}& \left\|[\, \tilde{p } \,] - [\, p \,] -  \gd p \right\|_{\GS}\\
\le{} & C \sum_{I=\pm}\left\|N_2^{I}(\tilde{\vv}^{I}) - N_2^{I}(  \vv^{I}) - (N_2^{I})_{\vv^{I}} (  \vv^{I})\cdot\gd \vv^{I} \right\|_{2,\ga;( \gb); \QQ^{I}}^{(-\ga;\EE)}\\
\le{} & C \sum_{I=\pm}\Big(\left\|N_2^{I}(\tilde{\vv}^{I}) - N_2^{I}(  \vv^{I}) -(N_2^{I}(  \vv^{I}))_{\vv^{I}} \cdot ( \tilde{\vv}^{I} -  \vv^{I})\right\|_{2,\ga;( \gb); \QQ^{I}}^{(-\ga;\EE)} \\
&+ \left\|(N_2^{I}(  \vv^{I}))_{\vv^{I}} \cdot(\tilde{\vv}^{I} - \vv^{I} - \gd \vv^{I}) \right\|_{2,\ga;( \gb); \QQ^{I}}^{(-\ga;\EE)}\Big)\\
\le{} & C  \sum_{I=\pm}\left( \left( \left\|\tilde{\vv}^I - \vv^I \right\|_{2,\ga;( \gb); \QQ^I}^{(-\ga;\EE)}\right)^2 +	\left\|\tilde{\vv}^{I} - \vv^{I} - \gd \vv^{I}\right\|_{2,\ga;( \gb); \QQ^{I}}^{(-\ga;\EE)}  \right)\\
\le{} & C (|\gd \ve| + \|\gd  w\|_{\GS})^2.
\end{align*}
This implies that $D\TT(\ve, w) (\gd \ve, \gd  w)$ is the differential of $\TT$ at $(\ve, w) $.

To show $\TT$ is $C^1$, we need to prove that $D\TT $ is continuous in  $(\ve, w) \in  (-\ve_1,\ve_1)\X V^{\ve_1} $.
Note that for any $(\tilde{\ve}, \tilde{ w})$ near $(\ve, w)$ and  $( \gd \ve,  \gd  w)\in (-\ve_1,\ve_1) \X \GS_{0}$,   $\tilde{\vv}^+$ solves \eqref{eqn-v1v2} \eqref{eqn-N+} with boundary data $\tilde{g}^+$.
Thus, $\tilde{\vv}^+ - \vv^+$ satisfies equation \eqref{eqn-phit-phi} with boundary data $\tilde{g}^+ - g^+ $. Denote
\begin{align*}
\kappa_1 &:= |\tilde{\ve} -  \ve| + \|\tilde{ w} -   w\|_{\GS},\\
\kappa_2 &:= |\gd \ve| + \|\gd   w\|_{\GS}.
\end{align*}
It follows from \eqref{est-v} in Proposition \ref{prop-estimate-gdvb1} that
	\begin{align}\label{est-phit-phi3}
  \|\tilde{\vv}^+ - \vv^+ \|_{2,\ga;( \gb); \QQ^+}^{(-\ga;\EE)}  \le C
  \|(\tilde{g}^+)' - (g^+)'\|_{2,\ga;(\gb); \R^+}^{(-\ga; \{0,1\})}\le C\kappa_1.
	\end{align}
Note that $\gd \tilde{\vv}^+ $ and  $\gd  \vv^+ $  take the same boundary conditions  and satisfy the equation \eqref{eqn-DT} with the coefficients  $(N^{+}_i)(\tilde{\vv}^{+})_{v^{+}_j}$ and  $(N^{+}_i)({\vv}^{+})_{v^{+}_j},$  respectively. Therefore,  $\gd \tilde{\vv}^+ -\gd  \vv^+ $solves the following equation
\begin{align*}
 \left((N^+_i)_{\vv^+}(\vv^+ ) \cdot  (\gd \tilde{\vv}^+ - \gd \vv^+ ) \right)_{y_i}
&=  \left(\left((N^+_i)_{\vv^+}(\tilde{\vv}^+ ) - (N^+_i)_{\vv^+}(\vv^+ ) \right) \cdot \gd \tilde{\vv}^+  \right)_{y_i}
\end{align*}
with homogeneous boundary conditions.
Noticing the fact that
\begin{align*}
 \|\gd \tilde{\vv}^+ \|_{2,\ga;( \gb); \QQ^+}^{(-\ga;\EE)}  \le C
  \|(\gd {g}^+)'\|_{2,\ga;(\gb); \R^+}^{(-\ga; \{0,1\})}   \le C\kappa_2,
\end{align*}
one can get
\begin{align*}
 \|\gd \tilde{\vv}^+ - \gd \vv^+ \|_{2,\ga;( \gb); \QQ^+}^{(-\ga;\EE)} &\le C
  \|\tilde{\vv}^+ - \vv^+ \|_{2,\ga;( \gb); \QQ^+}^{(-\ga;\EE)}    \|\gd \tilde{\vv}^+ \|_{2,\ga;( \gb); \QQ^+}^{(-\ga;\EE)}    \\
 & \le C \kappa_1 \kappa_2.
\end{align*}
Therefore,  the following estimate can be derived
\begin{align*}
&\| (D\TT(\tilde{\ve}, \tilde{ w}) - D\TT (\ve,  w)) (\gd \ve, \gd  w)\|_{\GS}\\
={}& \| \gd \tilde{p }  -  \gd p \|_{\GS}\\
\le{} &C  \|(N_2^{\pm})_{\vv^{\pm}} (  \tilde{\vv}^{\pm})\cdot  \gd \tilde{\vv}^{\pm}   - (N_2^{\pm})_{\vv^{\pm}} (  \vv^{\pm})\cdot  \gd \vv^{\pm} \|_{2,\ga;( \gb); \QQ^+}^{(-\ga;\EE)} \\
\le{} & C \left(\|\gd \tilde{\vv}^{\pm} - \gd \vv^{\pm} \|_{2,\ga;( \gb); \QQ^{\pm}}^{(-\ga;\EE)}
+ \| \tilde{\vv}^{\pm} -   \vv^{\pm} \|_{2,\ga;( \gb); \QQ^{\pm}}^{(-\ga;\EE)}  \|\gd \tilde{\vv}^{\pm} \|_{2,\ga;( \gb); \QQ^{\pm}}^{(-\ga;\EE)} \right)
 \\
  \le{} & C \kappa_1 \kappa_2,
\end{align*}
which implies that
\begin{align*}
\| D\TT(\tilde{\ve}, \tilde{ w}) - D\TT (\ve,  w)\|  \le C \kappa_1 = C (|\tilde{\ve} -  \ve| + \|\tilde{ w} -   w\|_{\GS}).
\end{align*}
Thus, $D\TT$ is Lipschitz continuous in $(-\ve_1,\ve_1)\X V^{\ve_1} $.

{\bf Isomorphism of $D_{ w} \TT (0,\mathbf{0})$.} We need to show that for any $\gd p \in \GS$, there exists a unique $\gd  w \in \GS_0$ such that $ D_{ w} \TT (0,\mathbf{0}) \gd  w = \gd p $.

Observe that $D_{ w} \TT (0,\mathbf{0})\gd  w =D \TT (0,\mathbf{0})(0, \gd w)   $.
When $\ve=0,  w= \mathbf{0} , \gd \ve=0$ in the procedure of defining $D\TT$ above, then $\gd \vv^{\pm}$ solve the following problems:
\begin{align}\label{prb-iso}
&\begin{cases}
a^{\pm}  ( \gd v_1^{\pm})_{  y_1} + b^{\pm}  (\gd v_2^{\pm})_{  y_2}=0  \\
   ( \gd v_1^{\pm})_{  y_2} -  (\gd v_2^{\pm})_{  y_1}=0 \end{cases}   \quad \text{ in } \QQ^{\pm};\\
&\begin{cases}
\gd v_1^{\pm} =(\gd  g)^{\prime}  & \text{  on } L^{+}\\
\gd v_2^{\pm} =0  & \text{  on } H^{\pm},
\end{cases} \label{con-iso}
\end{align}
where
\begin{align*}
&a^{\pm}  = (N^{\pm}_1)_{v_1^{\pm}}(\vv^{\pm}_0),  \quad
b^{\pm}  = (N^{\pm}_2)_{v_2^{\pm}}(\vv^{\pm}_0),\\
&\gd g (t)    = \begin{cases}
0, & 0\le t \le 1;\\
 \int_1 ^t  \gd w (s) \dx s, & t>1.
\end{cases}
\end{align*}
We stretch $\gd \vv^{\pm}$ in the $y_2$ direction and also flip $\gd \vv^-$ into $\QQ^+$ in the following way:
\begin{align*}
&\gd \vb_1^+ (y_1,y_2)   := \gd  v_1^+\left(y_1,  \sqrt{\tfrac{b^+}{a^+}} y_2\right), &&
\gd \vb_2^+ (y_1,y_2)    :=  \sqrt{\tfrac{b^+}{a^+}}\gd  v_2^+\left(y_1,  \sqrt{\tfrac{b^+}{a^+}} y_2\right),\\
&\gd \vb_1^- (y_1,y_2)    := \gd  v_1^-\left(y_1, -\sqrt{\tfrac{b^-}{a^-}} y_2\right), &&
\gd \vb_2^- (y_1,y_2)   :=- \sqrt{\tfrac{b^-}{a^-}}\gd  v_2^-\left(y_1, -\sqrt{\tfrac{b^-}{a^-}} y_2\right).
\end{align*}
Therefore,  each of the pairs $(\gd \vb_1^+,   \gd \vb_2^+ )$ and $(\gd \vb_1^-,   \gd \vb_2^- )$  satisfies  the Laplace equation with the same boundary conditions. By the uniqueness, we conclude that $\gd \vvb^+ = \gd \vvb^-$ in $\QQ^+$.

Note that
\begin{align*}
\gd p  &= (b^+  \gd v_2^+  -b^-  \gd v_2^-   )|_{ (1,\infty)\X \{0\}}= (\sqrt{a^+b^+}  +\sqrt{a^-b^-}) \gd \vb_2^+ |_{ (1,\infty)\X \{0\}}.
\end{align*}
To prove the isomorphism of $D_{ w} \TT (0,\mathbf{0})$, it suffices to show that the following problem is uniquely solvable for any given $ \gd p \in \GS$:
\begin{align}\label{prob-iso2}
&\begin{cases}
(\gd \vb_1^+)_{y_1} +(\gd \vb_2^+)_{y_2}  =0 \\
(\gd \vb_1^+)_{y_2} -(\gd \vb_2^+)_{y_1}  =0
\end{cases}  \quad  \text{ in }   \QQ^+;
\\
&\begin{cases}
\gd \vb_2^+ =0 & \text{ on }  H^+,  \\
 \gd \vb_1^+   =0 & \text{ on }  (0,1]\X \{0\},  \\
\gd \vb_2^+  =  \frac{\gd p  }{ \sqrt{a^+b^+}  +\sqrt{a^-b^-}}  & \text{ on } (1,\infty)\X \{0\}.\label{con-iso2}
\end{cases}
\end{align}

  We first  solve
\begin{align}
\GD(  \gd \vb_2^+  )=0  \label{eqn-gdphibar}
\end{align}
in $\QQ^+$
with the mixed boundary condition
\begin{align}\label{con-bdgdphiba}
   \begin{cases}
   \gd \vb_2^+  =0  & \text{ on }   H^+;\\
   (\gd \vb_2^+ )_{y_2} =0 & \text{ on }  (0,1)\X \{0\} ;\\
 \gd \vb_2^+  =  \frac{\gd p  }{ \sqrt{a^+b^+}  +\sqrt{a^-b^-}}  & \text{ on } (1,\infty)\X \{0\}.
\end{cases}
\end{align}
Problem \eqref{eqn-gdphibar}-\eqref{con-bdgdphiba} can be solved in the truncated domain and then taken limit as in Section \ref{sec-solvephi}.  Using the  barrier function
\begin{align*}
\vf_5 (\yy):= C \|\gd p \|_\GS\, r^{-\gb}\sin (\gb_1\theta + \theta_0)
\end{align*}
with
\begin{align*}
\gb < \gb_1 <1, \theta_0 \in (\tfrac{\pi}{2},\pi),\gb_1 \pi/2 +\theta_0 < \pi 
\end{align*}
for the decay estimate and the barrier function
\begin{align*}
\vf_6(\yy):= C \|\gd p \|_\GS\, (r^*)^{\ga}\sin (\gb_2\theta^* + \theta_0^*)
\end{align*}
with
\begin{align*}
\ga < \gb_2 <\frac{1}{2}, \theta^*_0 \in (\tfrac{\pi}{2},\pi),\gb_2 \pi +\theta_0^* < \pi 
\end{align*}
for the corner estimate near $T$, 
 we can derive the estimate
 	\begin{align}\label{est-dvbar2}
 \|\gd \vb^+_2   \|_{2,\ga;( \gb); \QQ^+}^{(-\ga;\{T\})} \le C \|\gd p \|_\GS.
 \end{align}

\begin{remark}
Due to the boundary  conditions at $O$  in \eqref{con-bdgdphiba} and using the reflection,  it is easy to see that the point $O$ is an interior point in the new domain and one thus can show  that  $\gd \vb^+_2 $ is $C^{2,\ga}$ up to the  point $O$.
\end{remark}

Next, we solve the following problem
\begin{align}\label{prob-isov1}
& \GD(  \gd \vb_1^+  )=0
 \quad  \text{ in }   \QQ^+; \\
&\begin{cases}
 \gd \vb_1^+= G(y_2):= -\int_{y_2}^\infty (\gd \vb_2^+)_{y_1} (0,s)\dx s  & \text{ on }  H^+,  \\
\gd \vb_1^+   =0 & \text{ on }  (0,1]\X \{0\},  \\
 (\gd \vb_1^+)_{y_2} =   (\gd \vb_2^+)_{y_1},   & \text{ on } (1,\infty)\X \{0\}.
\end{cases} \label{con-isov1}
\end{align}
It follows from   \eqref{est-dvbar2} that {$G(y_{2})$} defined in  \eqref{con-isov1} satisfies
\begin{align*}
&| G(y_2)| \le  C \|\gd p \|_\GS \, y_2^{-\gb} \quad \text{ for }  y_2>0,
\end{align*}
and $G^{\prime}(y_{2})=(\gd \vb_2^+)_{y_1} (0,y_{2}).$
It follows from  the same arguments as in Section \ref{sec-solvephi} that  $\gd \vb_1^+$ solves the problem \eqref{prob-isov1}-\eqref{con-isov1}  satisfying
\begin{align}
\|\gd \vb_1^+ \|_{2,\ga;( \gb); \QQ^+}^{(-\ga;\EE)} \le C \|\gd p \|_\GS. \label{est-gdv1}
\end{align}
Combining \eqref{est-dvbar2} with \eqref{est-gdv1} leads to
\begin{align}
\|\gd \vvb^+   \|_{2,\ga;( \gb); \QQ^+}^{(-\ga;\EE)} \le C \|\gd p \|_\GS. \label{est-gdv+gdp}
\end{align}

Next, we show that the solution $\gd\vb_2^+$ to the problems \eqref{eqn-gdphibar}-\eqref{con-bdgdphiba} and the solution $\gd\vb_1^+$ to the problems \eqref{prob-isov1}-\eqref{con-isov1} will lead to the solution to the original problem \eqref{prob-iso2}-\eqref{con-iso2}. Indeed, set
\begin{align*}
 \begin{cases}
V:= (\gd \vb_1^+)_{y_1} +(\gd \vb_2^+)_{y_2},   \\
W:=(\gd \vb_1^+)_{y_2} -(\gd \vb_2^+)_{y_1}.
 \end{cases}
\end{align*}
Then the system \eqref{eqn-gdphibar}, \eqref{prob-isov1} is equivalent to the following system
\begin{align} \label{eqn-VW}
\begin{cases}
V_{y_1} +W_{y_2}=0   \\
V_{y_2} - W_{y_1}=0
 \end{cases}   \quad  \text{ in }\ \QQ^+
\end{align}
and conditions (5.13) and \eqref{con-isov1} imply that
\begin{align} \label{con-VW}
\begin{cases}
V=0  & \text{ on }\   (0,1)\X \{0\}  \\
W=0   & \text{ on }\   H^+ \cup (1,\infty)\X \{0\}.
 \end{cases}
\end{align}
The system \eqref{eqn-VW} implies that there exists a harmonic potential $\Phi$ in $\QQ^+$ such that
\begin{align*}
\GD \Phi = 0,  \quad  \Phi(O) = 0, \quad \grad \Phi = (V,W).
\end{align*}
Conditions \eqref{con-VW} become
\begin{align} \label{con-Phi}
\begin{cases}
\Phi =0  & \text{ on }\    H^+ \cup (0,1]\X \{0\}  \\
\Phi_{y_2}=0   & \text{ on }\   (1,\infty)\X \{0\}.
\end{cases}
\end{align}
It follows from (5.18) that $\Phi$ is H\"{o}lder continuous up to  $O$ and $T$ . Therefore, the solution is unique and $\Phi =0$.
This implies that $V=W=0$ in $\QQ^+$.

Set $\gd w = \gd \vb_1^+ |_{[1,\infty)\X \{0\}}$. Then \eqref{est-gdv1}
shows that $\gd w \in \GS_{0}$.
Hence, we have shown that there exists a unique $\gd  w  \in \GS_{0} $ such that
$D_{w} \TT (0,\mathbf{0}) \gd  w = \gd p$.
This completes the proof of the isomorphism of  $ D_{w} \TT (0,\mathbf{0}) $.

Then by the implicit function theorem, there exists $\ve_0 >0$, such that $\TT (\ve, w) = 0$ is uniquely solvable for $0\le \ve \le \ve_0$ and $w(\ve) $ is $C^1$ on $[0,\ve_0]$ with
\begin{align}\label{est-wve}
\|w(\ve)\|_\GS \le C\ve.
\end{align}
Since $g^+$ is piecewise defined by $\ve\zeta^+$ on $[0,1]$ and $w(\ve)$ on $(1,\infty)$, the estimate \eqref{est-wve} implies that
\begin{align}\label{est-zeta+}
 \|(g^+)' \|_{2,\ga;(\gb); \R^+}^{(-\ga; \{0,1\})} \le C\ve.
\end{align}
Choosing $\ve_0$ suitably small so that $C\ve_0 \le \tau$, then \eqref{est-zeta+} implies the condition \eqref{con-g+small}. Hence,  the estimate  \eqref{est-gdv+} in Proposition \ref{prop-nonlinear}, together with \eqref{est-zeta+}, gives that
	\begin{equation}
\|U^{+} - U^+_0	\|_{2,\ga;( \gb); \QQ^+}^{(-\ga;\EE)} \le C\|\vv^+ - \vv^+_0  \|_{2,\ga;( \gb); \QQ^+}^{(-\ga;\EE)} \le C \ve. \label{est-U-U0}
	\end{equation}
Similar estimate holds in the domain $\QQ^-$.		
Thus, we have solved the {\bf reduced problem} with the following estimate:
	\begin{equation}
	\|U^{\pm} - U^{\pm}_0	\|_{2,\ga;( \gb); \QQ^+}^{(-\ga;\EE)}  + \| g_T\|_{3,\ga;( \gb-1); (1,\infty)}^{(-1-\ga ;\{1\})} \le C \ve. \label{est-U-U0+-}
	\end{equation}

Since the coordinate transformation \eqref{def-coord} is invertible and bi-Lipschitz,   the solution transformed back in $\xx$-coordinates solves the {\bf main problem} and the estimate \eqref{est-U-U0+-} implies the estimate \eqref{est-Ugc} in the {\bf Main Theorem}.  This completes the proof of  the {\bf Main Theorem}.

\begin{proof}[Proof of Theorem \ref{thm-stable}]
	Since $w(\ve)$ is a $C^1$ map on $[0,\ve_0]$, it follows that for any $\ve, \tilde{\ve} \in [0,\ve_0]$,
\begin{align}\label{est-stabw}
\|w(\ve) - w(\tilde{\ve} ) \|_\GS \le \sup_{\ve\in [0,\ve_0] } \|w'(\ve)\| |\ve - \tilde{\ve}|\le C|\ve - \tilde{\ve}|.
\end{align}	
This implies the estimate \eqref{est-stab}, hence proves Theorem \ref{thm-stable}.	
\end{proof}

\section*{Acknowledgments}
Jun Chen's research was supported in part by Yichun University Doctoral Start-up Grant 207-3360119008. The work of Aibin Zang was supported in part  by the National Natural Science Foundation of China (Grant no. 11771382). The research of Zhouping Xin was supported in part by Zheng Ge Ru Foundation and by Hong Kong RGC Earmarked Research Grants, CUHK14305315, CUHK14302819, CUHK14300917, and CUHK14302917.

\end{document}